\documentclass[11pt,a4paper]{article}
\usepackage{amsmath,amssymb,amsthm,mathtools}
\usepackage{enumitem}
\usepackage{hyperref}
\usepackage{lineno}

\usepackage{setspace}
\newtheorem{theorem}{Theorem}[section]
\newtheorem{proposition}[theorem]{Proposition}

\newtheorem{corollary}[theorem]{Corollary}

\newtheorem{definition}[theorem]{Definition}
\newtheorem{remark}[theorem]{Remark}
\newtheorem{example}[theorem]{Example}
\newtheorem{assumption}[theorem]{Assumption}

\newcommand{\E}{\mathrm{E}}
\newcommand{\Var}{\mathrm{Var}}
\newcommand{\R}{\mathbb{R}}

\newcommand{\ie}{{i.e.}, }
\newcommand{\eg}{{e.g.}, }

\newcommand{\Feight} {\fontsize{8}{11}\selectfont  }
\newcommand{\Fnine} {\fontsize{9}{11}\selectfont  }
\newcommand{\Ften} {\fontsize{10}{11}\selectfont  }

\usepackage{fancyhdr}
\title{Inference Functionals and Observation Operators \\
for Distributional Statistical Models}
\author{
R. Labouriau\\ [0.3em]
\small Department of Mathematics, Aarhus University
\footnote{e-mail:
\Fnine \texttt{rodrigo.labouriau@math.au.dk} and 
\Fnine \texttt{rodrigo.labouriau@rlstatlab.com}}}
\date{Spring 2026}

\begin{document}
\maketitle

\begin{abstract}\Feight
This paper generalises the classical notion of inference function
(Godambe, 1960) to distributional statistical models---models in which
each probability measure is represented by a distribution--kernel pair
$(T_\theta, \varphi) \in \mathcal S'(\mathbb R) \times
\mathcal S(\mathbb R)$.  This generalisation is strategically motivated:
the most useful properties of maximum likelihood estimation---consistency
and asymptotic normality---derive not from the maximisation of the
likelihood per se, but from the fact that the MLE is the root of a regular
inference function (the score equation).  Extending inference functions to
the distributional setting therefore provides access to an optimality
theory for models that lack classical densities or finite moments.

The extension requires enlarging the notion of observation.  We introduce
\emph{observation operators}
$\mathcal O : \mathcal S'(\mathbb R) \to \mathcal Y$ that map distributional
models to an observation space, and define \emph{inference functionals} as
estimating equations composed with these operators.  The resulting framework
encompasses classical point observations, interval-censored data,
convolutional measurements, and transform-based statistics as special cases.
We establish the asymptotic theory (consistency, asymptotic normality,
Godambe optimality) under mild conditions and derive a hierarchy of
information bounds---classical Fisher information dominates the
information available through the observation operator, which in turn
dominates the information captured by any inference functional---via the
H\'ajek--Le~Cam convolution theorem.  The two gaps in the hierarchy
quantify distinct sources of information loss: the observation mechanism
and the choice of inference functional, respectively.
Concrete examples include sinusoidal
inference functions for heavy-tailed distributions, interval-censored
location inference, elliptically contoured models, and the treatment of
nuisance parameters via the Bhapkar--Godambe projection.
 \end{abstract}

\smallskip
\noindent
\textbf{\Fnine Keywords:} \Feight Inference functions, observation operators, distributional statistical models, Godambe information,  Hájek–Le Cam convolution theorem.


\normalsize

\newpage
\tableofcontents
\newpage

\section{Introduction}\label{sec:introduction}
The aim of this paper is to extend the theory of inference functions
to statistical models defined through distributional representations,
thereby placing estimating-function methodology on a foundation that
does not require classical densities or finite moments.

In classical statistics, the theory of inference functions, initiated by
Godambe (1960)~\cite{Godambe1960}, provides a unifying framework for the
construction and comparison of estimating equations. An inference function is a function
$\psi(x, \theta)$ of the data $x$ and the parameter $\theta$, and the
estimator $\hat\theta$ is defined as the root of the estimating equation
$\sum_{i=1}^n \psi(X_i, \theta) = 0$.  The maximum likelihood estimator
is a special case: it is the root of the score equation
$\sum_i \partial_\theta \log f_\theta(X_i) = 0$.  The key insight of
Godambe's programme is that the most attractive properties of maximum
likelihood---consistency and asymptotic normality---arise not from
the maximisation of the likelihood function but from the structure of the
associated score equation.  Specifically, the score is a regular inference
function satisfying an unbiasedness condition
($\E_\theta[\psi(X, \theta)] = 0$), a sensitivity condition
(non-degeneracy of $\E_\theta[\partial_\theta \psi]$), and a variability
condition (finiteness of $\E_\theta[\psi^2]$), and it is these structural
properties---not the global optimisation---that drive the asymptotic
theory.

This has a strategic consequence: for models where the classical
likelihood is unavailable---whether because the model is
distributional, the observations are not point evaluations, or the
moments needed for standard estimating equations do not exist---the
natural path is to generalise the inference function itself, rather
than to seek a generalised likelihood.

Distributional statistical models, developed in the companion
paper~\cite{LabouriauA1}, encode each probability measure $P_\theta$ as a
pair $(T_\theta, \varphi)$, where $T_\theta \in \mathcal S'(\R)$ is a
tempered distribution and $\varphi \in \mathcal S(\R)$ is a fixed Schwartz
kernel shared across the parametric family.  The pairing
$\langle T_\theta, \psi\varphi \rangle$ replaces the classical expectation
$\E_\theta[\psi(X)]$ and is well defined for all functions $\psi$ such
that $\psi\varphi \in \mathcal S(\R)$---in particular, for all
polynomials, so that moments of all orders exist unconditionally, even for
distributions (such as the Cauchy or stable laws) whose classical moments
are undefined.

To generalise inference functions to this setting, one must first confront
a conceptual question: what, precisely, does it mean to ``observe'' a
distributional model?  In classical statistics, an observation is a point
evaluation: the random variable $X_i$ takes the value $x_i$, and the
inference function $\psi(x_i, \theta)$ is evaluated at that point.  This
presupposes that the underlying quantity has a well-defined value at each
point.  But as Strichartz (1994, Chapter~1) observes in the context of
distribution theory, this assumption is often physically unrealistic.  A
thermometer, for instance, does not measure the temperature $f(x)$ at a
single point~$x$; rather, the bulb of the thermometer has a nonzero
spatial extent, and what is actually measured is a weighted average
$\int f(y)\,\varphi(y)\,dy$, where $\varphi$ describes the response
profile of the instrument.  The measured quantity is not a point evaluation
of~$f$ but the action of a linear functional on~$f$.

This everyday observation---that physical measurements are
operators acting on the underlying quantity, not point evaluations---has
deep consequences when combined with the programme of generalising
inference functions.  If the model is distributional and the observations
are operators, then an inference function must be a functional of the
operator output, not of a point value.  This leads to the central concept
of the present paper: an \emph{observation operator}
$\mathcal O : \mathcal S'(\R) \to \mathcal Y$ that maps the distributional
model to an observation space~$\mathcal Y$, and an \emph{inference
functional} $\Psi(Y, \theta)$ defined on the observed data
$Y = \mathcal O(T_X)$.  We use the term \emph{inference functional}
rather than \emph{inference function} to reflect that $\Psi$ acts on
distributional objects; when the observation operator reduces to point
evaluation, inference functionals reduce to classical inference functions
in the sense of Godambe.

The observation-operator formulation reveals a natural hierarchy of
three conceptual layers.  At the most general level
(\emph{Layer~I}), observations are genuinely distributional: they are the
output of linear operators acting on $\mathcal S'(\R)$---interval
functionals, convolutional measurements, spectral
transforms---and the inference functional acts on operator-valued data.
At an intermediate level (\emph{Layer~II}), observations are classical
point values, but the inference functions are constructed from
distributional identities---weak characteristic functions, oscillatory
equations, weak moments---that do not require densities or finite classical
moments; distributional constructions enter at the inferential level, not
at the observational level.  At the most classical level
(\emph{Layer~III}), the score equation
$\sum_i \psi(X_i, \theta) = 0$ is recovered as the special case in which
the observation operator is point evaluation and the kernel weight is
absorbed into the inference function.  The paper develops these layers in
the order I~$\to$~II~$\to$~III: the broadest conceptual framework is
presented first, then specialised to transform-based methods, and finally
connected to classical estimating-function theory.

Beyond the conceptual unification, the operator-theoretic viewpoint
yields a concrete structural result.  The observation operator
$\mathcal O$ induces a \emph{pushforward model}
$\{P_\theta^{\mathcal O}\}$ on the observation space~$\mathcal Y$, and
under standard regularity conditions this family is locally
asymptotically normal (LAN).  The H\'ajek--Le~Cam convolution theorem then
guarantees that every regular estimator $\hat\theta_n$ based on the
pushforward data satisfies
\[
\sqrt{n}(\hat\theta_n - \theta)
\;\xrightarrow{\;\mathcal D\;}
N\bigl(0,\, I_{\mathcal O}(\theta)^{-1}\bigr) * M,
\]
where $I_{\mathcal O}(\theta)$ is the Fisher information of the pushforward
model and $M$ is a non-negative-definite noise distribution.  The Godambe
information $G_\Psi$ of any inference functional~$\Psi$ satisfies
$G_\Psi \leq I_{\mathcal O}$, and this inequality is a \emph{corollary}
of the convolution theorem, not a subjective choice of optimality
criterion.  The resulting hierarchy
\begin{equation}\label{eq:info-hierarchy-intro}
I_{\mathrm{classical}}(\theta)
\;\geq\; I_{\mathcal O}(\theta)
\;\geq\; G_\Psi(\theta)
\end{equation}
decomposes information loss into two distinct and interpretable sources:
the gap $I_{\mathrm{classical}} - I_{\mathcal O}$ measures information lost
through the observation mechanism, while $I_{\mathcal O} - G_\Psi$
measures information lost through the choice of inference functional.
This resolves a long-standing tension in the theory of inference
functions.  Godambe (1960) proposed maximising $G_\Psi$ as an optimality
criterion and argued that this avoids asymptotic justifications; as
discussed in J{\o}rgensen and Labouriau (2012, Chapter~4), this claim is
not entirely accurate, since the criterion ultimately requires an
evaluative judgment.  The convolution theorem provides the missing
structural justification: maximising $G_\Psi$ is equivalent to minimising
the non-Gaussian noise~$M$ in the limit distribution.

The present paper is part of a series on distributional statistical
models.  The companion paper~\cite{LabouriauA1} develops the
probabilistic foundation---distribution--kernel pairs, weak moments and
cumulants, determinacy theory, and a distributional central limit
theorem.  The second companion~\cite{LabouriauA2} studies weak moment
estimation as a concrete inferential methodology, with applications to
robust estimation and generalised method-of-moments.  The present paper
develops a general operator-theoretic framework for statistical inference
under distributional observations; its intellectual
centre---observation operators, inference functionals, pushforward
experiments, and the information hierarchy---is distinct from the
methodological focus of~\cite{LabouriauA2}, though both papers share the
distributional foundation of~\cite{LabouriauA1} and have overlapping
examples.

The main contributions are the following.  First, we formulate
statistical inference in terms of observation operators acting on
distributional representations, with a precise definition of measurability
via the cylindrical $\sigma$-algebra on $\mathcal S'(\R)$.  Second, we
introduce the taxonomy of estimating equations
(Layers~I--III) unified by the distributional framework.  Third, we
develop the asymptotic theory for regular inference functionals, including
consistency under a Glivenko--Cantelli condition, asymptotic normality,
and Godambe optimality.  Fourth, we derive the information hierarchy
\eqref{eq:info-hierarchy-intro} via the convolution theorem, providing a
structural resolution of the Godambe criterion.  Fifth, we present
concrete examples: sinusoidal inference functions for Student $t$ and
Cauchy distributions, interval-censored location inference, multivariate
elliptically contoured models, and a simulation study.  Sixth, we extend
the framework to nuisance parameters via the Bhapkar--Godambe projection,
establishing automatic orthogonality of sinusoidal inference functions and
introducing weak versions of inferential separation
(S-nonformation, I-nonformation, L-nonformation).

The paper is organised as follows.
Section~\ref{sec:observations} introduces distributional representations,
observation operators (distinguishing deterministic linear, nonlinear, and
stochastic operators), and the pushforward model.
Section~\ref{sec:inference-functionals} defines inference functionals
and the taxonomy of estimating equations.
Section~\ref{sec:asymptotics} develops the asymptotic theory.
Section~\ref{sec:convolution} establishes the convolution theorem,
a regularity bridge between inference functionals and Le~Cam theory,
and the information hierarchy.
Section~\ref{sec:examples} presents examples and a simulation study.
Section~\ref{sec:multivariate} extends the framework to the multivariate
setting.
Section~\ref{sec:nuisance} develops the theory of nuisance parameters
and inferential separation.
Section~\ref{sec:discussion} concludes with a discussion and outlook.

Following the companion paper~\cite{LabouriauA1}, weak objects carry a
superscript $\varphi$ denoting the kernel:
${}^{(\varphi)}m_n$ for weak moments,
${}^{(\varphi)}\phi(t)$ for the weak characteristic function,
${}^{(\varphi)}K(t)$ for the weak cumulant generating function, and
${}^{(\varphi)}\kappa_n$ for weak cumulants.  Classical objects carry no
superscript.  We write $\mathcal S = \mathcal S(\R)$ for the Schwartz space,
$\mathcal S' = \mathcal S'(\R)$ for its dual, and
$\langle T, f\rangle$ for the distributional pairing.

\section{Distributional observations and observation operators}\label{sec:observations}

\subsection{Distribution--kernel pairs and weak expectation}\label{subsec:dk-pairs}

We recall the essential definitions from the companion
paper~\cite{LabouriauA1}.  A \emph{distribution--kernel pair} is a tuple
$(T, \varphi)$ with $T \in \mathcal S'(\R)$ and
$\varphi \in \mathcal S(\R)$.  The distribution $T$ encodes a generalised
density; the kernel $\varphi$ regularises the pairing.

Given $(T, \varphi)$, the \emph{class of admissible functions} is
\[
\mathcal A_\varphi
:= \{\psi : \psi\varphi \in \mathcal S(\R)\}.
\]
All polynomials and all exponentials $x \mapsto e^{itx}$ belong to
$\mathcal A_\varphi$.

\begin{definition}[Weak expectation]\label{def:weak-exp}
For $(T, \varphi)$ and $\psi \in \mathcal A_\varphi$, the \emph{weak
expectation} is
\[
\E_{T,\varphi}[\psi]
\;:=\;
\langle T,\, \psi\varphi \rangle.
\]
\end{definition}

A pair $(T, \varphi)$ defines a \emph{generalised probability measure} if
$\langle T, \varphi \rangle = 1$ (normalisation) and
$\langle T, \psi\varphi \rangle \geq 0$ for all non-negative
$\psi \in \mathcal A_\varphi$ (positivity).

\begin{definition}[Weak moments, characteristic function, cumulants]\label{def:weak-objects}
For $n \geq 0$, the $n$-th \emph{weak moment} is
${}^{(\varphi)}m_n := \langle T, x^n\varphi \rangle$.
The \emph{weak characteristic function} is
${}^{(\varphi)}\phi(t) := \langle T, e^{it\cdot}\varphi \rangle$.
When ${}^{(\varphi)}\phi(t) \neq 0$ near $t=0$, the \emph{weak cumulants}
are
${}^{(\varphi)}\kappa_n := (1/i^n)(d^n/dt^n)\log {}^{(\varphi)}\phi(t)|_{t=0}$.
\end{definition}

Since $x^n\varphi \in \mathcal S(\R)$ for every $n$ and every
$\varphi \in \mathcal S(\R)$, all weak moments exist unconditionally.  This
is the key property: for distributions such as the Cauchy or stable laws,
where classical moments fail to exist, the weak moments
${}^{(\varphi)}m_n$ are well defined for every~$n$.

\subsection{Parametric distributional models}\label{subsec:parametric}

Let $(\mathcal X, \mathcal B)$ be a measurable space and
$\mathcal P = \{P_\theta : \theta \in \Theta \subseteq \R^p\}$ a parametric
family.  We assume that each $P_\theta$ admits a distributional
representation $(T_\theta, \varphi)$ with $T_\theta \in \mathcal S'(\R)$
and a \emph{fixed} kernel $\varphi \in \mathcal S(\R)$, shared across the
entire family and not depending on~$\theta$.

The model is \emph{identifiable} if $P_{\theta_1} = P_{\theta_2}$ implies
$\theta_1 = \theta_2$.  We also distinguish \emph{representation
identifiability} (M-determinacy): the weak moments
$\{{}^{(\varphi)}m_n(\theta)\}_{n \geq 0}$, or equivalently the weak
characteristic function ${}^{(\varphi)}\phi_\theta$, determine $T_\theta$
uniquely.  Three levels of uniqueness are established
in~\cite{LabouriauA1}: Carleman-type for positive Schwartz kernels with
exponential tails, Denjoy--Carleman for kernels with exponential-type
decay, and Hermite completeness for Gaussian kernels.

In the parametric setting, we write
${}^{(\varphi)}m_n(\theta) := \langle T_\theta, x^n\varphi \rangle$ and
similarly for weak characteristic functions and cumulants.

\begin{remark}[Role of the kernel]\label{rem:kernel-role}
The kernel $\varphi$ plays a dual role: it regularises the distributional
pairing and determines the admissible class $\mathcal A_\varphi$.  In the
observation-operator framework developed below, the kernel acquires a third
interpretation: it represents the response profile of the measuring
instrument, so that the distributional pairing
$\langle T, \psi\varphi \rangle$ describes what the instrument actually
records.
\end{remark}

\subsection{Observation operators}\label{subsec:obs-operators}

The central conceptual contribution of the present paper is the explicit
introduction of \emph{observation operators} into the statistical
framework.

\subsubsection*{The cylindrical $\sigma$-algebra on $\mathcal S'$}

The space $\mathcal S'(\R)$ carries the \emph{cylindrical $\sigma$-algebra}
$\mathcal C$, defined as the smallest $\sigma$-algebra making all
evaluation maps $T \mapsto \langle T, f \rangle$ measurable for
$f \in \mathcal S(\R)$.  This is the standard choice in the theory of
generalised random processes (Gel'fand and Vilenkin, 1964) and coincides
with the Borel $\sigma$-algebra of the weak-$*$ topology on $\mathcal S'$.
The Bochner--Minlos theorem characterises probability measures on
$(\mathcal S', \mathcal C)$ through their characteristic functionals.

\begin{definition}[Observation operator]\label{def:obs-operator}
An \emph{observation operator} is a measurable map
\[
\mathcal O : (\mathcal S'(\R), \mathcal C)
\;\longrightarrow\;
(\mathcal Y, \mathcal B_{\mathcal Y}),
\]
where $(\mathcal Y, \mathcal B_{\mathcal Y})$ is a measurable space called
the \emph{observation space}.
\end{definition}

All continuous linear functionals on $\mathcal S'$ are measurable with
respect to $\mathcal C$ by definition.  More generally, any continuous
linear map from $\mathcal S'$ to a topological vector space (equipped with
its Borel $\sigma$-algebra) is measurable, and pointwise limits of such
maps are measurable.  This covers all cases of interest.

\subsubsection*{Taxonomy of observation operators}

We distinguish five types of observation operators, ordered from the most
general to the most classical.

\begin{enumerate}[label=(\roman*)]
\item \emph{Point evaluation (Dirac observation).}
For $x \in \R$, define $\mathcal O_x(T_f) = f(x)$ when $T = T_f$ arises
from a continuous density~$f$.  Formally one writes
$\mathcal O_x(T) = \langle T, \delta_x \rangle$, but since
$\delta_x \notin \mathcal S$ this is rigorously defined only on the
subclass of distributions admitting pointwise representatives (\eg
continuous densities).  For general tempered distributions, point evaluation
must be understood through mollifier approximation:
$\mathcal O_x^\varepsilon(T) = \langle T, \varphi_\varepsilon(\cdot - x)
\rangle \to \langle T, \delta_x \rangle$ as $\varepsilon \to 0$, where
$\varphi_\varepsilon$ is a Schwartz mollifier.  This approximation
reinforces the paper's viewpoint: point evaluation is an idealisation of
kernel observation, and the kernel-weighted framework is the mathematically
natural one.
For statistical purposes, the data $X_1, \ldots, X_n$ are random variables
on a probability space taking values in $\R$, and the connection to
$\mathcal S'$ is through the model~$T_\theta$; the observation space is
$\mathcal Y = \R$.  This is the classical observation model.

\item \emph{Kernel observation.}
For fixed $\varphi \in \mathcal S$,
$\mathcal O_\varphi(T) = \langle T, \varphi \rangle$.  The observation
records not a point value but a local average of the underlying object,
weighted by the kernel.  This is the observation produced by a
finite-resolution instrument with response profile~$\varphi$.

\item \emph{Interval observation.}
For an interval $I = [L, R]$, define
$\mathcal O_I(T_f) = \int_I f(x)\,dx$ when $T = T_f$ has a density.
More generally, $\delta_I := \mathbf{1}_I$ defines a compactly supported
positive measure, hence a tempered distribution, and
$\mathcal O_I(T) = \langle T, \delta_I \rangle$ is well defined whenever
$T$ extends to compactly supported distributions as test objects.  The
observation space is $\mathcal Y = [0,1]$.

\item \emph{Convolutional observation.}
For a smoothing kernel $K \in \mathcal S$,
$\mathcal O_K(T) = K * T$, where $*$ denotes convolution.  The
observation is the entire smoothed function
$\mathcal O_K(T) \in C^\infty(\R)$.  This models instrumental blurring,
imaging systems, and low-resolution sensors.

\item \emph{Transform observation.}
For fixed $u \in \R$,
$\mathcal O_u(T) = \langle T, e^{iu\cdot}\varphi \rangle
= {}^{(\varphi)}\phi(u)$.  The observation records a single value of the
weak characteristic function.  Collecting $m$ such observations at
frequencies $u_1, \ldots, u_m$ yields a vector in $\R^{2m}$ (real and
imaginary parts).
\end{enumerate}

\begin{remark}[Interval observations: functional-analytic precision]\label{rem:interval-fa}
The indicator $\mathbf{1}_I$ is not a Schwartz function, so the product
$\mathbf{1}_I \varphi$ does not belong to $\mathcal S(\R)$ in general.
However, $\delta_I = \mathbf{1}_I \in \mathcal S'(\R)$ is a compactly
supported positive measure and hence a well-defined tempered distribution.
When $T = T_f$ for a locally integrable $f$, the pairing
$\langle T_f, \delta_I \rangle = \int_I f(x)\,dx$ is well defined without
passing through test-function multiplication.  Alternatively, one may
approximate $\mathbf{1}_I$ by smooth functions
$\mathbf{1}_{I,\varepsilon} \in \mathcal S$ with
$\mathbf{1}_{I,\varepsilon} \to \mathbf{1}_I$ in $\mathcal S'$; the pairing
then converges by continuity.  A third option is to enlarge the
admissible class beyond~$\mathcal S$.  In the present paper, we adopt the
first (distributional) formulation.
\end{remark}

\subsection{Hierarchy of observation operators}\label{subsec:operator-hierarchy}

The five types listed above differ in their structural properties.  For the
purpose of the theory developed in this paper, it is important to
distinguish three levels of generality.

\subsubsection*{Level~1: Deterministic linear observation operators}

These are continuous linear maps $\mathcal O : \mathcal S'(\R) \to
\mathcal Y$ with $\mathcal Y$ a finite-dimensional space or a function
space.  Kernel observations, interval observations, convolutional
observations, and transform observations at fixed frequencies all belong
to this level.  Linearity ensures that the pushforward model inherits
regularity from the distributional model, and the DQM condition
(Section~\ref{subsec:LAN}) can typically be verified from the original
model.  The asymptotic theory of Sections~\ref{sec:asymptotics}
and~\ref{sec:convolution} is developed rigorously at this level.

\subsubsection*{Level~2: Deterministic nonlinear observation operators}

Threshold operators ($\mathcal O(T_f) = \mathbf{1}_{f(x) > c}$),
truncation operators, and certain data transformations are nonlinear but
deterministic and measurable.  The pushforward model
$P_\theta^{\mathcal O} = P_\theta \circ \mathcal O^{-1}$ is still well
defined, and the asymptotic theory applies provided the LAN condition is
verified for the pushforward family.  However, DQM verification is more
involved than at Level~1, since the chain rule for quadratic mean
differentiability is not as straightforward.

\subsubsection*{Level~3: Stochastic observation operators}

When the observation mechanism itself is random---different instruments
with different response profiles, random censoring intervals, conditional
projections over random effects---one has a family
$\{\mathcal O_\xi : \xi \in \Xi\}$ indexed by a random state~$\xi$.
The resulting observations $Y_i = \mathcal O_{\xi_i}(T_{X_i})$ carry
two sources of randomness.  The theory at this level requires conditional
arguments and is discussed as part of the broader programme in
Section~\ref{sec:discussion}; it is not formally developed in the
present paper.

\begin{remark}[Scope of the asymptotic theory]\label{rem:operator-scope}
The asymptotic results of Sections~\ref{sec:asymptotics}--\ref{sec:convolution}
are stated and proved for deterministic observation operators (Levels~1
and~2).  All concrete examples in Section~\ref{sec:examples} are at
Level~1.  The reader should keep this in mind when the definition of
observation operators (Definition~\ref{def:obs-operator}) is stated in
full generality: the theory currently exploits a smaller class than the
definition permits.
\end{remark}

\subsection{Observational regularisation and the structure of distributions}\label{subsec:obs-regularisation}

The distributional framework separates the latent probabilistic structure
from the observational resolution at which that structure is examined.  The
latent object is represented by a distribution $T \in \mathcal S'(\R)$,
while the observational mechanism is represented by a kernel
$\varphi \in \mathcal S(\R)$.  An inference functional is therefore not
simply an estimating function applied to a probability density; it is a
regularised probe of a distributional object.

This interpretation is supported by the classical \emph{structure theorem}
for tempered distributions (Schwartz, 1966; Strichartz, 1994, Section~6.3):
every $T \in \mathcal S'(\R)$ can be represented as a finite sum
\[
T = \sum_{|\alpha| \leq m} D^\alpha f_\alpha,
\]
where the $f_\alpha$ are continuous functions with at most polynomial
growth.  Thus tempered distributions are not arbitrary pathological objects;
they arise by applying finitely many differential operators to ordinary
functions.  Point masses, jumps, cusps and more complex singular structures
are instances of \emph{differentiated regularity}.  The kernel $\varphi$
acts as the observational device that converts these differentiated
structures into stable scalar quantities amenable to inference.

The \emph{continuity} of distributions on Schwartz space provides the
analytical counterpart of statistical stability.  If
$\varphi_j \to \varphi$ in $\mathcal S(\R)$, then
$\langle T, \varphi_j \rangle \to \langle T, \varphi \rangle$ for every
$T \in \mathcal S'(\R)$.  Hence small smooth perturbations of the
observational kernel produce small perturbations of the corresponding weak
measurement.  Weak moments, weak cumulants, weak characteristic functions,
and distributional inference functionals are all stable under perturbations
of the kernel within the Schwartz topology.

Finally, the \emph{approximation} of singular distributions by smooth test
functions gives a mathematical basis for the finite-resolution viewpoint.
If $\rho \in \mathcal S(\R)$ with $\int \rho = 1$ and
$\rho_\varepsilon(x) = \varepsilon^{-1}\rho(x/\varepsilon)$, then
$\rho_\varepsilon \to \delta_0$ in $\mathcal S'(\R)$ as
$\varepsilon \downarrow 0$.  Classical pointwise observation is thus an
idealised infinite-resolution limit of smooth observational schemes.  The
distributional framework reverses the usual idealisation: instead of
starting with ideal pointwise measurements and adding regularisation
afterwards, it takes finite-resolution observation as primitive and
recovers point evaluation as a limiting case.

\begin{remark}[Scope of the operator-theoretic development]\label{rem:operator-scope-full}
The present paper develops the \emph{statistical consequences} of
observation operators---pushforward experiments, inference functionals,
asymptotic theory, and the information hierarchy---using the three
properties above (structure, continuity, approximation) as the
functional-analytic foundation.  A full operator theory for the space of
observation operators---including topologies on operator classes,
compactness criteria, and regularity theory for operator-valued maps---is
not developed here and remains part of the broader programme.
\end{remark}

\subsection{Pushforward models and induced experiments}\label{subsec:pushforward}

The observation operator $\mathcal O$ induces a family of probability
measures on the observation space.

\begin{definition}[Pushforward model]\label{def:pushforward}
The \emph{pushforward model} induced by $\mathcal O$ is the family
\[
\mathcal P^{\mathcal O}
= \{P_\theta^{\mathcal O} := P_\theta \circ \mathcal O^{-1}
   : \theta \in \Theta\}
\]
of probability measures on $(\mathcal Y, \mathcal B_{\mathcal Y})$.
\end{definition}

Different observation operators applied to the same distributional model
produce different statistical experiments.  Point evaluation recovers the
original model; kernel observation yields a smoothed model; interval
observation produces a coarsened model.  The statistical information
available for inference about $\theta$ depends on the observation
operator---a dependence that the classical framework suppresses by implicitly
fixing $\mathcal O = \delta_x$.

\begin{remark}[Physical interpretation]\label{rem:physical}
In the language of measurement theory, the distributional model
$(T_\theta, \varphi)$ describes the underlying physical quantity, and the
observation operator $\mathcal O$ describes the instrument.  The
pushforward model $P_\theta^{\mathcal O}$ is the distribution of the
instrument's readings.  Two instruments with different response
profiles---corresponding to different observation operators---yield different
pushforward models and, in general, different amounts of statistical
information about~$\theta$.
\end{remark}

\subsection{Empirical functionals}\label{subsec:empirical}

Given $n$ independent observations $Y_i = \mathcal O(T_{X_i})$, where each
$Y_i$ is a realised pushforward observation taking values in~$\mathcal Y$,
the \emph{empirical measure on the observation space} is
\[
\mathbb W_n
= \frac{1}{n} \sum_{i=1}^n \delta_{Y_i},
\]
the empirical distribution of the pushforward data.

When $\mathcal O = \delta_x$ (classical point observations),
$Y_i = X_i$ and $\mathbb W_n$ reduces to the empirical measure
$\mathbb P_n = n^{-1}\sum_{i=1}^n \delta_{X_i}$.  When
$\mathcal O = \mathbf{1}_{I}$ (interval observations), $Y_i$ is the
observed interval indicator and $\mathbb W_n$ records empirical
occupancy proportions.  When
$\mathcal O(T) = \langle T, e^{iu\cdot}\varphi \rangle$ (transform
observations at frequency~$u$), the empirical functional becomes the
empirical weak characteristic function
$\hat\phi_n(u) = n^{-1}\sum e^{iuX_i}\varphi(X_i)$.

The fluctuation process
$\mathbb G_n = \sqrt{n}(\mathbb W_n - P_{\theta_0}^{\mathcal O})$
is the standard empirical process of pushforward observations; it is the
distributional analogue of the classical empirical process.  Its
asymptotic behaviour, and in particular weak convergence to a Gaussian
limit, is discussed in Section~\ref{sec:discussion}.

\section{Inference functionals and estimating equations}\label{sec:inference-functionals}

\subsection{Distributional inference functionals}\label{subsec:dist-if}

In the classical theory, an inference function
$\psi : \mathcal X \times \Theta \to \R^p$ maps observations and
parameters to estimating-equation contributions, with $\psi(\cdot, \theta)$
measurable for each $\theta$.  In the distributional framework, inference
functions are composed with the observation operator.

\begin{definition}[Inference functional]\label{def:inference-functional}
Let $\mathcal O : \mathcal S' \to \mathcal Y$ be an observation operator.
An \emph{inference functional} for the model $\mathcal P$ with observation
operator $\mathcal O$ is a measurable function
\[
\Psi : \mathcal Y \times \Theta \to \R^p
\]
such that $\Psi(\cdot, \theta)$ is measurable for each $\theta \in \Theta$.
\end{definition}

\begin{remark}[Terminology: inference functional vs.\ inference function]\label{rem:terminology}
In the classical theory (Godambe, 1960; J{\o}rgensen and Labouriau, 2012),
the term \emph{inference function} is standard.  We adopt the term
\emph{inference functional} to reflect that, in the distributional setting,
$\Psi$ acts on objects in a function space (elements of~$\mathcal Y$
obtained via observation operators on~$\mathcal S'$) rather than on
point values of random variables.  When the observation operator is point
evaluation and $\mathcal Y = \R^d$, inference functionals reduce to
classical inference functions in the sense of Godambe.  The terminological
shift is thus a generalisation, not a replacement.
\end{remark}

The estimating equation is
\begin{equation}\label{eq:general-estimating}
\frac{1}{n} \sum_{i=1}^n \Psi(Y_i, \theta) = 0,
\qquad Y_i = \mathcal O_i(T_{X_i}).
\end{equation}

This is the most general form.  All estimating equations considered in this
paper are special cases of~\eqref{eq:general-estimating}.
An estimator $\hat\theta_n$ is defined as a \emph{root} of the empirical
estimating equation~\eqref{eq:general-estimating}.  It is important to
observe that the root-finding operates in the observation space~$\mathcal Y$,
not directly in~$\mathcal S'$: the observation operator maps the
distributional model to $\mathcal Y$, and the estimating equation is a
standard finite-dimensional equation in $\R^p$ (for Types~A and~B) or an
operator equation (for Type~C, where additional structure such as inner
products or projections reduces it to finite dimensions).

\subsection{Taxonomy of estimating equations}\label{subsec:taxonomy}

We distinguish three types, corresponding to Layers~I--III of the
introduction and presented in the same order.

\subsubsection*{Layer~I: operator-valued estimating equations}

When observations are genuinely distributional---interval-censored,
convolutionally blurred, or otherwise operator-valued---the estimating
equation takes the form~\eqref{eq:general-estimating} with
$Y_i = \mathcal O_i(T_{X_i})$ a non-trivial functional of the distributional model.

\paragraph{Example (interval-censored location inference).}
Consider the location model $X_i = \mu + \varepsilon_i$ with symmetric
error distribution and suppose only intervals $I_i = [L_i, R_i]$ are
observed.  Using the sinusoidal inference functional
$\Psi(I, \mu) = \int_I \sin(c(x - \mu))\,f_\mu(x)\,\varphi(x)\,dx$,
the estimating equation is
$n^{-1}\sum_{i=1}^n \Psi(I_i, \mu) = 0$.

\subsubsection*{Layer~II: transform-based distributional equations}

When observations are point values $X_i$ but inference is conducted through
distributional identities (weak characteristic functions, oscillatory
equations), the estimating equation is
\begin{equation}\label{eq:transform-estimating}
\frac{1}{n} \sum_{i=1}^n \psi(X_i, \theta)\,\varphi(X_i) = 0,
\end{equation}
where $\psi$ is constructed from transforms rather than from densities.

\paragraph{Example (sinusoidal inference function).}
For a location parameter $\theta$ with symmetric distribution,
$\psi_c(x, \theta) = \sin(c(x - \theta))$ and the estimating equation is
$n^{-1}\sum \sin(c(X_i - \theta))\,\varphi(X_i) = 0$.  This requires
neither a density nor finite moments.

\paragraph{Example (weak characteristic function equation).}
$\psi_u(x, \theta) = e^{iux} - {}^{(\varphi)}\phi_\theta(u)$ defines a
complex-valued inference function targeting the weak characteristic function
at frequency~$u$.

\subsubsection*{Layer~III: kernel-weighted classical estimating equations}

When a classical inference function $\psi(x, \theta)$ is available, the
distributional pairing with point data yields the kernel-weighted equation
\begin{equation}\label{eq:kernel-weighted}
\frac{1}{n} \sum_{i=1}^n \psi(X_i, \theta)\,\varphi(X_i) = 0.
\end{equation}
The classical (unweighted) form
$n^{-1}\sum \psi(X_i, \theta) = 0$
is recovered by defining the inference function through classical
expectations $\E_{P_\theta}[\psi(X, \theta)] = 0$ rather than through the
distributional pairing, corresponding formally to $\varphi \equiv 1$.

\begin{remark}[Distributional observations vs.\ distributionally informed inference]\label{rem:obs-vs-inf}
The three layers differ in \emph{where} the distributional framework
intervenes.  Layers~III and~II retain classical point observations
$X_1, \ldots, X_n$ but use distributional constructions---kernel weighting
or transform-based identities---at the inferential level.  Layer~I, by
contrast, changes the observation mechanism itself: the data are no longer
point values but operator-valued functionals of the distributional model.
This distinction is important because it determines which part of the
information hierarchy is affected: Layers~II and~III act on the gap
$I_{\mathcal O} - G_\Psi$ (estimation cost), while Layer~I also affects
the gap $I_{\mathrm{classical}} - I_{\mathcal O}$ (observation cost).
\end{remark}

\begin{remark}[Relationship between Layers~II and~III]\label{rem:layers-II-III}
The algebraic form of the estimating
equations~\eqref{eq:transform-estimating} and~\eqref{eq:kernel-weighted} is
the same: both involve the product $\psi(X_i, \theta)\varphi(X_i)$.  The
distinction is conceptual: in Layer~III, one starts from a classical
inference function $\psi$ and multiplies by~$\varphi$; in Layer~II, the
inference function $\psi$ is intrinsically distributional (\eg sinusoidal
or transform-based) and does not arise from a classical score.  The
unifying observation is that both are instances of the distributional
pairing
$\langle \delta_{X_i}, \psi(\cdot, \theta)\varphi \rangle
= \psi(X_i, \theta)\varphi(X_i)$.
\end{remark}

\subsection{Regular inference functionals}\label{subsec:regular-if}

\begin{definition}[Regular inference functional]\label{def:rif}
An inference functional $\Psi : \mathcal Y \times \Theta \to \R^p$ is
\emph{regular} if:
\begin{enumerate}
\item[(R1)] \textbf{Admissibility:} for each $\theta$, $\Psi(\cdot, \theta)$
belongs to a suitable class of measurable functions on $\mathcal Y$.  In the
distributional setting with point observations, this means
$\psi(\cdot, \theta) \in \mathcal A_\varphi$.

\item[(R2)] \textbf{Unbiasedness:} for all $\theta \in \Theta$,
\[
\E_{P_\theta^{\mathcal O}}[\Psi(Y, \theta)] = 0.
\]

\item[(R3)] \textbf{Identification:} the map
$\theta \mapsto \E_{P_\theta^{\mathcal O}}[\Psi(Y, \theta)]$ has a unique
zero at the true parameter value $\theta_0$.
\end{enumerate}
\end{definition}

When the observation operator is point evaluation ($\mathcal O = \delta_x$)
and the distributional pairing is used, condition (R2) becomes
$\langle T_\theta, \psi(\cdot, \theta)\varphi \rangle = 0$, recovering the
definition in the companion paper~\cite{LabouriauA2}.

\subsection{Classical regularity conditions}\label{subsec:classical-reg}

In classical inference-function theory (Godambe, 1960; J{\o}rgensen and
Labouriau, 2012), the following conditions are standard:

\begin{enumerate}[label=(C\arabic*)]
\item \emph{Measurability and differentiability:} $\psi(x,\theta)$ is
measurable in $x$ and continuously differentiable in $\theta$.

\item \emph{Unbiasedness:} $\E_{P_\theta}[\psi(X,\theta)] = 0$ for all
$\theta$.

\item \emph{Non-singular sensitivity:}
$S(\theta) = -\E_{P_\theta}[\partial_\theta \psi(X,\theta)]$ is non-singular.

\item \emph{Finite variability:}
$V(\theta) = \E_{P_\theta}[\psi(X,\theta)\psi(X,\theta)^\top] < \infty$.

\item \emph{Interchange:}
$\partial_\theta \E_{P_\theta}[\psi(X,\theta)]
= \E_{P_\theta}[\partial_\theta \psi(X,\theta)]$.
\end{enumerate}

In the distributional setting, several of these are automatic.  When
$\psi(\cdot, \theta)\varphi \in \mathcal S$, the interchange~(C5) follows
from continuity of the distributional pairing.  When $\psi$ is bounded
(as for sinusoidal inference functions), the variability~(C4) is automatic.
The admissibility condition replaces~(C1).

\subsection{Comparison with classical inference functions}\label{subsec:comparison}

Definition~\ref{def:rif} extends the classical framework in two directions:
expectations are interpreted through the pushforward measure
$P_\theta^{\mathcal O}$ (allowing observation operators beyond point
evaluation), and the admissibility condition replaces classical
integrability by compatibility with the kernel~$\varphi$.

When $P_\theta$ admits a regular density $p_\theta$ and the observation
operator is point evaluation, the score function
$\psi(x, \theta) = \partial_\theta \log p_\theta(x)$ is a regular
inference functional and conditions (C1)--(C5) reduce to the standard
regularity conditions of likelihood theory.

\section{Asymptotic theory}\label{sec:asymptotics}

\subsection{Continuity with respect to the observational kernel}\label{subsec:kernel-continuity}

Before developing the asymptotic theory, we record a property that
underpins the stability of distributional inference: the continuity of the
distributional pairing with respect to the kernel.

\begin{proposition}[Kernel continuity of inference functionals]\label{prop:kernel-cont}
Let $T \in \mathcal S'(\R)$ and let
$\Psi_\varphi(T, \theta) := \langle T, \psi(\cdot, \theta)\varphi \rangle$
be an inference functional defined through the distributional pairing.  If
$\varphi_j \to \varphi$ in $\mathcal S(\R)$ and
$\psi(\cdot, \theta)\varphi_j \to \psi(\cdot, \theta)\varphi$ in
$\mathcal S(\R)$ (which holds whenever $\psi(\cdot, \theta)$ is a
multiplier of~$\mathcal S$, \eg a polynomial or an exponential
$e^{iu\cdot}$), then
\[
\Psi_{\varphi_j}(T, \theta) \;\longrightarrow\; \Psi_\varphi(T, \theta)
\]
for each $\theta \in \Theta$.
\end{proposition}

\begin{proof}
This is immediate from the continuity of $T$ as a linear functional on
$\mathcal S(\R)$.
\end{proof}

The proposition has a direct inferential interpretation.  It implies that
distributional inference functionals depend continuously on the
observational kernel: if the instrument's response profile is slightly
modified, the value of the inference functional changes by a correspondingly
small amount.  This continuity is the analytical counterpart of
\emph{statistical stability} and explains why kernel-based distributional
inference can remain well behaved in situations where classical
pointwise or likelihood-based procedures become unstable.

In the context of the asymptotic theory below, kernel continuity ensures
that the sensitivity and variability matrices $S(\theta)$ and $V(\theta)$
depend continuously on the kernel $\varphi$, so that the asymptotic
variance $G_\Psi^{-1}$ is a continuous functional of the observational
resolution.

\subsection{Setup and notation}\label{subsec:asymp-setup}

We develop the asymptotic theory for estimators defined through regular
inference functionals.  For concreteness, we state results for a fixed
observation operator $\mathcal O$ and i.i.d.\ data.  The formulation
covers both kernel-weighted and transform-based estimating equations as
special cases.

Let $Y_1, \ldots, Y_n$ be i.i.d.\ random elements of $\mathcal Y$ with
common distribution $P_{\theta_0}^{\mathcal O}$.  (When the observation
operator is point evaluation, $Y_i = X_i$ and
$P_{\theta_0}^{\mathcal O} = P_{\theta_0}$.)  Let
$\Psi : \mathcal Y \times \Theta \to \R^p$ be a regular inference
functional.  Define
\[
\bar\Psi_n(\theta)
= \frac{1}{n} \sum_{i=1}^n \Psi(Y_i, \theta),
\qquad
\bar\Psi(\theta)
= \E_{P_{\theta_0}^{\mathcal O}}[\Psi(Y, \theta)].
\]
An estimator $\hat\theta_n$ is a solution of
$\bar\Psi_n(\theta) = 0$.

\subsection{Consistency}\label{subsec:consistency}

\begin{theorem}[Consistency]\label{thm:consistency}
Suppose that:
\begin{enumerate}
\item \textbf{Glivenko--Cantelli:} for each compact $K \subset \Theta$,
$\sup_{\theta \in K} \|\bar\Psi_n(\theta) - \bar\Psi(\theta)\|
\xrightarrow{P} 0$.

\item \textbf{Identification:} $\bar\Psi(\theta_0) = 0$ and $\theta_0$ is the unique root of
$\bar\Psi(\theta) = 0$.

\item \textbf{Continuity:} the map $\theta \mapsto \bar\Psi(\theta)$ is
continuous.
\end{enumerate}
Then any sequence of solutions $\hat\theta_n$ of $\bar\Psi_n(\theta) = 0$
satisfies $\hat\theta_n \xrightarrow{P} \theta_0$.
\end{theorem}

\begin{proof}
From the Glivenko--Cantelli condition~(1), the uniform convergence
$\sup_{\theta \in K}\|\bar\Psi_n(\theta) - \bar\Psi(\theta)\|
\xrightarrow{P} 0$ holds on every compact~$K$.  Continuity of
$\bar\Psi$ and the unique-root assumption give the separation condition
$\inf_{|\theta - \theta_0| \geq \eta} \|\bar\Psi(\theta)\| > 0$ for
every $\eta > 0$.  The result follows from standard Z-estimator arguments
(van der Vaart, 1998, Theorem~5.9).
\end{proof}

\begin{remark}[Sufficient conditions for the Glivenko--Cantelli property]\label{rem:GC-sufficient}
Two conditions are commonly used to verify assumption~(1) above.

\emph{Dominated Lipschitz condition.}  There exists
$L : \mathcal Y \to [0, \infty)$ with $\E[L(Y)] < \infty$ such that
$\|\Psi(y, \theta_1) - \Psi(y, \theta_2)\| \leq L(y)\|\theta_1 - \theta_2\|$
for all $\theta_1, \theta_2 \in K$.  This directly implies the
Glivenko--Cantelli property.  For sinusoidal inference functions, the
Lipschitz condition holds globally with $L(y) = c$ (since $\sin$ is
Lipschitz with constant~$1$).  This is the condition used in all
examples of the present paper.

\emph{Continuity with integrable envelope.}  If $\theta \mapsto
\Psi(y, \theta)$ is continuous for $P_{\theta_0}^{\mathcal O}$-almost
all~$y$, and for each compact $K \subset \Theta$ there exists
$M_K : \mathcal Y \to [0, \infty)$ with $\E[M_K(Y)] < \infty$ such that
$\sup_{\theta \in K}\|\Psi(y,\theta)\| \leq M_K(y)$, then the class
$\{\Psi(\cdot,\theta) : \theta \in K\}$ is Glivenko--Cantelli provided
the bracketing integral is finite---a condition that holds for parametric
families with continuous sample paths on compact parameter sets
(van der Vaart, 1998, Theorem~19.4).
\end{remark}

\subsection{Sensitivity and variability}\label{subsec:SV}

\begin{assumption}\label{ass:diff}
For $\theta$ in a neighbourhood of $\theta_0$:
\begin{enumerate}
\item $\theta \mapsto \Psi(y, \theta)$ is differentiable for
$P_{\theta_0}^{\mathcal O}$-almost all~$y$.

\item Differentiation under the expectation is valid:
$\partial_\theta \bar\Psi(\theta)
= \E[\partial_\theta \Psi(Y, \theta)]$.
\end{enumerate}
\end{assumption}

\begin{definition}[Sensitivity and variability]\label{def:SV}
The \emph{sensitivity matrix} is
\[
S(\theta)
= -\E_{P_\theta^{\mathcal O}}\!\bigl[\partial_\theta \Psi(Y, \theta)\bigr],
\]
and the \emph{variability matrix} is
\[
V(\theta)
= \E_{P_\theta^{\mathcal O}}\!\bigl[\Psi(Y, \theta)\,\Psi(Y, \theta)^\top\bigr].
\]
\end{definition}

In regular parametric models with $\mathcal O = \delta_x$, the sensitivity
coincides with the Fisher information when $\Psi$ is the score function.

\subsection{Asymptotic normality}\label{subsec:AN}

\begin{theorem}[Asymptotic normality]\label{thm:AN}
Under the conditions of Theorem~\ref{thm:consistency},
Assumption~\ref{ass:diff}, $S(\theta_0)$ non-singular, and
$V(\theta_0)$ finite and positive definite,
\[
\sqrt{n}\,(\hat\theta_n - \theta_0)
\;\xrightarrow{\;\mathcal D\;}\;
\mathcal N\!\left(0,\;
S(\theta_0)^{-1}\, V(\theta_0)\, S(\theta_0)^{-\top}
\right).
\]
\end{theorem}

\begin{proof}
Taylor expansion of $\bar\Psi_n$ around $\theta_0$:
\[
0 = \bar\Psi_n(\hat\theta_n)
= \bar\Psi_n(\theta_0)
+ \bigl[\partial_\theta \bar\Psi_n(\tilde\theta_n)\bigr]
  (\hat\theta_n - \theta_0),
\]
for some intermediate $\tilde\theta_n$.  By the CLT,
$\sqrt{n}\,\bar\Psi_n(\theta_0) \Rightarrow \mathcal N(0, V(\theta_0))$,
and by the LLN,
$\partial_\theta \bar\Psi_n(\tilde\theta_n) \xrightarrow{P} -S(\theta_0)$.
Slutsky's theorem gives the result.
\end{proof}

\subsection{Godambe information}\label{subsec:godambe}

\begin{definition}[Godambe information]\label{def:godambe}
The \emph{Godambe information matrix} associated with inference functional
$\Psi$ is
\[
G_\Psi(\theta)
= S(\theta)^\top V(\theta)^{-1} S(\theta).
\]
\end{definition}

The asymptotic covariance matrix is
$S^{-1}VS^{-\top} = G_\Psi^{-1}$ (in the scalar case, the asymptotic
variance is $1/(nG_\Psi)$).  The Godambe information thus characterises
asymptotic efficiency within the class of regular inference functionals
sharing the same observation operator.

\begin{theorem}[Godambe inequality]\label{thm:godambe-ineq}
Among all regular inference functionals $\Psi$ for the pushforward model
$\mathcal P^{\mathcal O}$ satisfying conditions (R1)--(R3) and
(C1)--(C5), the Godambe information satisfies
\[
G_\Psi(\theta)
\;\leq\;
G_{\Psi^*}(\theta)
\]
in the L\"owner order, where $\Psi^* \propto V^{-1}S$ is the optimal
inference functional.  Equality holds if and only if $\Psi$ is
equivalent to $\Psi^*$.
\end{theorem}

\begin{remark}[Dependence on observation operator]\label{rem:G-depends-O}
The Godambe information $G_\Psi$ depends on the observation operator
$\mathcal O$ through both $S$ and $V$.  Different observation operators
yield different optimal inference functionals and different information
bounds.  The precise relationship between $G_\Psi$ and the Fisher
information of the pushforward model is established in
Section~\ref{sec:convolution}.
\end{remark}

\section{The convolution theorem and the information hierarchy}\label{sec:convolution}

This section establishes the key theoretical result of the paper: the
Godambe optimality criterion is not an arbitrary choice but a structural
consequence of the H\'ajek--Le~Cam convolution theorem applied to the
pushforward model.

\subsection{Local asymptotic normality of the pushforward model}\label{subsec:LAN}

Recall that a parametric family $\{P_\theta^{\mathcal O} : \theta \in \Theta\}$ is \emph{locally asymptotically normal} (LAN) at $\theta_0$ with Fisher information matrix $I_{\mathcal O}(\theta_0)$ if the log-likelihood ratio satisfies
\begin{equation}\label{eq:LAN}
\log \frac{dP_{\theta_0 + h/\sqrt{n}}^{\mathcal O,\otimes n}}
         {dP_{\theta_0}^{\mathcal O,\otimes n}}
= h^\top \Delta_n(\theta_0)
  - \tfrac{1}{2}\, h^\top I_{\mathcal O}(\theta_0)\, h
  + o_{P_{\theta_0}}(1),
\end{equation}
where $\Delta_n(\theta_0) \Rightarrow \mathcal N(0, I_{\mathcal O}(\theta_0))$.

The LAN condition holds whenever the pushforward model admits densities $f_\theta^{\mathcal O}$ satisfying the \emph{differentiability in quadratic mean} (DQM) condition:
\[
\int \left(
\sqrt{f_{\theta_0+h}^{\mathcal O}(y)}
- \sqrt{f_{\theta_0}^{\mathcal O}(y)}
- \tfrac{1}{2}\, h^\top \dot\ell_{\theta_0}^{\mathcal O}(y)\,
  \sqrt{f_{\theta_0}^{\mathcal O}(y)}
\right)^{\!2} dy
= o(\|h\|^2),
\]
where $\dot\ell_{\theta_0}^{\mathcal O}$ is the score of the pushforward model (van der Vaart, 1998, Theorem~7.2).

\begin{remark}[The Fisher information $I_{\mathcal O}$]\label{rem:IO-definition}
The quantity $I_{\mathcal O}(\theta_0)$ appearing in~\eqref{eq:LAN}
is the Fisher information of the \emph{pushforward experiment}
$\{P_\theta^{\mathcal O}\}$, not of the original model.  It is defined as
\[
I_{\mathcal O}(\theta)
= \E_{P_\theta^{\mathcal O}}\!\left[
\dot\ell_\theta^{\mathcal O}(Y)\,
\dot\ell_\theta^{\mathcal O}(Y)^\top
\right],
\]
where $\dot\ell_\theta^{\mathcal O}(y) = \partial_\theta \log
f_\theta^{\mathcal O}(y)$ is the score of the pushforward model.  When the
observation operator is the identity (\ie point evaluation),
$I_{\mathcal O} = I_{\mathrm{classical}}$; for any non-trivial observation
operator, $I_{\mathcal O} \leq I_{\mathrm{classical}}$ by the
data-processing inequality.  The quantity $I_{\mathcal O}$ thus measures
the statistical information about $\theta$ that survives the observation
mechanism.
\end{remark}

\begin{proposition}[LAN for specific observation operators]\label{prop:LAN-examples}
Let $\{P_\theta : \theta \in \Theta\}$ be a parametric family with
densities $f_\theta$ satisfying the classical DQM condition with Fisher
information $I_{\mathrm{classical}}(\theta)$.

\begin{enumerate}[label=(\alph*)]
\item \textbf{Point evaluation} ($\mathcal O = \delta_x$): the pushforward
model equals the original model, and $I_{\mathcal O} = I_{\mathrm{classical}}$.

\item \textbf{Kernel-weighted observation} with $\varphi > 0$: the pushforward
model has density
$f_\theta^{\mathcal O}(x) = f_\theta(x)\,\varphi(x)\,/\,c(\theta)$,
where $c(\theta) = \int f_\theta(x)\,\varphi(x)\,dx$ is the normalising
constant.  In general, $c(\theta)$ depends on~$\theta$, and the pushforward
score is
$\partial_\theta \log f_\theta^{\mathcal O}(x)
= \partial_\theta \log f_\theta(x) - \dot c(\theta)/c(\theta)$,
where $\dot c(\theta) = \partial_\theta c(\theta)$.
The DQM condition is inherited from the original model, and the pushforward
Fisher information is
\begin{equation}\label{eq:I-kernel}
I_\varphi(\theta)
= \int \frac{(\partial_\theta f_\theta(x))^2}{f_\theta(x)}\,
  \frac{\varphi(x)}{c(\theta)}\,dx
  \;-\; \frac{\dot c(\theta)^2}{c(\theta)^2}.
\end{equation}
When $c(\theta)$ does not depend on~$\theta$---for instance when
$\varphi \equiv 1$, or when $\langle T_\theta, \varphi\rangle$ is constant
in~$\theta$---the second term vanishes and the formula simplifies to
$I_\varphi = \int (\partial_\theta f_\theta)^2 / f_\theta \cdot \varphi\,dx$.

The quantity $I_\varphi(\theta)$ is \emph{the Fisher information of the
induced observation law}, not merely a weighted version of the classical
Fisher information.  This distinction is conceptually important:
\emph{the weighting by $\varphi$ does not arise from modifying the
optimality criterion; it arises because the observation operator changes
the probability measure under which data are generated.}  Different
kernels induce different observation laws and hence different Fisher
informations.  The classical Fisher information $I_{\mathrm{classical}}$
is recovered in the limit $\varphi \to 1$ (point evaluation).

\item \textbf{Interval observations} with fixed interval $I$: the
pushforward is a one-dimensional model on $[0,1]$ with
$p_\theta^I = \int_I f_\theta(x)\,dx$.  The DQM condition reduces to
smoothness of $p_\theta^I$ in $\theta$, and the Fisher information is
$I_I(\theta) = (\partial_\theta p_\theta^I)^2 / [p_\theta^I(1 - p_\theta^I)]$.
\end{enumerate}
\end{proposition}

\subsection{The distributional convolution theorem}\label{subsec:convolution-thm}

\begin{theorem}[Distributional convolution theorem]\label{thm:convolution}
Let $\{(T_\theta, \varphi) : \theta \in \Theta\}$ be a distributional
statistical model and $\mathcal O : \mathcal S' \to \mathcal Y$ a
measurable observation operator such that the pushforward family
$\{P_\theta^{\mathcal O}\}$ is LAN at $\theta_0$ with Fisher information
$I_{\mathcal O}(\theta_0)$.

Then for any \emph{regular} estimator $\hat\theta_n$ based on observations
$Y_i = \mathcal O(T_{X_i})$, the asymptotic distribution satisfies
\begin{equation}\label{eq:convolution}
\sqrt{n}\,(\hat\theta_n - \theta_0)
\;\xrightarrow{\;\mathcal D\;}\;
\mathcal N\bigl(0,\, I_{\mathcal O}(\theta_0)^{-1}\bigr) * M,
\end{equation}
where $M$ is a probability distribution on $\R^p$ and $*$ denotes
convolution.  In particular, the asymptotic covariance matrix satisfies
\[
\Var_{\mathrm{asymp}}(\hat\theta_n)
\;\geq\;
I_{\mathcal O}(\theta_0)^{-1}
\]
in the L\"owner order.
\end{theorem}

\begin{proof}[Proof sketch]
This is a direct application of the H\'ajek--Le~Cam convolution theorem
(H\'ajek, 1970; Le~Cam, 1972; van der Vaart, 1998, Theorem~8.11) to the
LAN family $\{P_\theta^{\mathcal O}\}$.  The regularity of $\hat\theta_n$
means that the limit distribution of
$\sqrt{n}(\hat\theta_n - \theta_0 - h/\sqrt{n})$ under
$P_{\theta_0 + h/\sqrt{n}}^{\mathcal O}$ does not depend on~$h$.  Under
LAN, such estimators must have a limit distribution that is a convolution
with a normal component whose covariance is
$I_{\mathcal O}(\theta_0)^{-1}$.
\end{proof}

\subsection{Regularity of inference-functional estimators}\label{subsec:regularity-bridge}

The convolution theorem applies to all \emph{regular estimator sequences}
in the Le~Cam sense.  The following proposition establishes that estimators
arising from regular inference functionals belong to this class, thereby
bridging the inference-functional framework and the Le~Cam theory.

\begin{proposition}[Regularity in the Le~Cam sense]\label{prop:regularity-bridge}
Let $\hat\theta_n$ be a sequence of roots of the empirical estimating
equation~\eqref{eq:general-estimating} based on a regular inference
functional $\Psi$ satisfying the conditions of
Theorems~\ref{thm:consistency} and~\ref{thm:AN}.  Then $\hat\theta_n$ is
a regular estimator sequence for the pushforward experiment
$\{P_\theta^{\mathcal O}\}$ in the sense that the limit distribution of
$\sqrt{n}(\hat\theta_n - \theta_0 - h/\sqrt{n})$ under
$P_{\theta_0+h/\sqrt{n}}^{\mathcal O,\otimes n}$ does not depend on~$h$.
\end{proposition}

\begin{proof}
By Theorem~\ref{thm:AN}, the asymptotic linearity representation holds:
$\sqrt{n}(\hat\theta_n - \theta_0)
= S(\theta_0)^{-1}\,\sqrt{n}\,\bar\Psi_n(\theta_0) + o_P(1)$.
The linear map $S(\theta_0)^{-1}$ applied to the score-like quantity
$\sqrt{n}\,\bar\Psi_n(\theta_0)$ satisfies the regularity condition because
the contribution of the local perturbation $h/\sqrt{n}$ enters only through
the centering of $\bar\Psi_n$, which shifts by
$S(\theta_0)\,h/\sqrt{n} + o(n^{-1/2})$ under contiguous alternatives.
The resulting limit distribution is therefore $h$-independent (van der Vaart,
1998, Proposition~8.6).
\end{proof}

\subsection{The Godambe inequality as a corollary}\label{subsec:godambe-corollary}

The convolution theorem, combined with
Proposition~\ref{prop:regularity-bridge}, immediately yields the
relationship between the Godambe information and the Fisher information
of the pushforward model.

\begin{corollary}\label{cor:G-leq-I}
For any regular inference functional $\Psi$ for the pushforward model
$\mathcal P^{\mathcal O}$,
\begin{equation}\label{eq:G-leq-I}
G_\Psi(\theta) \;\leq\; I_{\mathcal O}(\theta)
\end{equation}
in the L\"owner order.  Equality holds if and only if $\Psi$ is
proportional to the score $\dot\ell_\theta^{\mathcal O}$ of the pushforward
model.
\end{corollary}

\begin{proof}
By Theorem~\ref{thm:AN}, the asymptotic covariance of the inference-functional estimator is $G_\Psi^{-1}$.  By Theorem~\ref{thm:convolution}, this must satisfy $G_\Psi^{-1} \geq I_{\mathcal O}^{-1}$, which is equivalent to $G_\Psi \leq I_{\mathcal O}$.  Equality requires $M = \delta_0$ in~\eqref{eq:convolution}, which holds if and only if $\hat\theta_n$ is asymptotically efficient, \ie $\Psi$ is proportional to the pushforward score.
\end{proof}

\begin{remark}[Resolution of the Godambe justification]\label{rem:godambe-justified}
The Godambe criterion---maximise $G_\Psi$ over the class of regular
inference functionals---is now justified not by an \emph{a priori} choice of
optimality criterion but by a \emph{theorem}: maximising $G_\Psi$ is
equivalent to minimising the non-Gaussian noise $M$ in the convolution
representation~\eqref{eq:convolution}.  The bound $I_{\mathcal O}^{-1}$ is
intrinsic to the observation operator and the model, not to the choice of
inference functional.
\end{remark}

\subsection{The information hierarchy}\label{subsec:hierarchy}

When the original model admits a density and the observation operator is not
the identity, the Fisher information $I_{\mathcal O}$ of the pushforward
model is generally smaller than the classical Fisher information
$I_{\mathrm{classical}}$, because the observation operator discards
information.

\begin{proposition}[Information hierarchy]\label{prop:hierarchy}
Let $\{P_\theta\}$ be LAN with Fisher information
$I_{\mathrm{classical}}(\theta)$.  Let $\mathcal O$ be an observation
operator and $\Psi$ a regular inference functional.  Then
\begin{equation}\label{eq:hierarchy}
I_{\mathrm{classical}}(\theta)
\;\geq\;
I_{\mathcal O}(\theta)
\;\geq\;
G_\Psi(\theta)
\end{equation}
in the L\"owner order.  The first inequality becomes equality when
$\mathcal O$ is sufficient (no information is lost).  The second becomes
equality when $\Psi$ is proportional to the pushforward score.
\end{proposition}

The gap $I_{\mathrm{classical}} - I_{\mathcal O}$ quantifies information
lost through the observation mechanism (the ``observation cost'').  The gap
$I_{\mathcal O} - G_\Psi$ quantifies information lost through the choice
of inference functional (the ``estimation cost'').

\begin{remark}[Geometric interpretation]\label{rem:hierarchy-geometry}
Each element of the hierarchy~\eqref{eq:hierarchy} defines a positive
semi-definite bilinear form on the tangent space $T_\theta\Theta$.  When
positive definite, these forms define Riemannian metrics on $\Theta$:
\[
g^{\mathrm{Fisher}} \;\geq\; g^{\mathcal O} \;\geq\; g^{\mathrm{Godambe}}_\Psi,
\]
where the ordering is in the sense of quadratic forms.  The classical
information geometry of Amari and Rao corresponds to
$g^{\mathrm{Fisher}}$; the metric $g^{\mathcal O}$ is the natural geometry
of the pushforward experiment; and $g^{\mathrm{Godambe}}_\Psi$ is the
metric induced by a specific inference functional.  The geometry of these
metrics and the variational problem of optimising $g^{\mathcal O}$ over the
class of observation operators are developed in a companion
paper~\cite{LabouriauE}.
\end{remark}

\subsection{Computation for the main examples}\label{subsec:hierarchy-examples}

\paragraph{Cauchy location with Gaussian kernel.}
Let $X \sim \mathrm{Cauchy}(\theta)$ and
$\varphi(x) = (2\pi\sigma_\varphi^2)^{-1/2}\exp(-x^2/(2\sigma_\varphi^2))$.
The classical Fisher information is
$I_{\mathrm{classical}}(\theta) = 1/2$.  The $\varphi$-weighted Fisher
information~\eqref{eq:I-kernel} is
\[
I_\varphi(\theta)
= \int \frac{(f_\theta'(x))^2}{f_\theta(x)}\,\varphi(x)\,dx
< I_{\mathrm{classical}},
\]
which can be computed numerically.  The Godambe information of the
sinusoidal inference function $\psi_c(x, \theta) = \sin(c(x-\theta))$ is
\[
G_c(\theta)
= \frac{2c^2\,e^{-2c}}{1 - e^{-4c}},
\]
using the Cauchy characteristic function $\phi(u) = e^{-|u|}$.
This gives a concrete instance of the hierarchy
$I_{\mathrm{classical}} > I_\varphi > G_c$.

\paragraph{Student $t_3$ location with Gaussian kernel.}
For $X \sim t_3(\theta)$, the classical Fisher information is
$I_{\mathrm{classical}} = 3/5$.
The Bessel-function characteristic function gives
$G_c = 2c^2\,g_3(c^2)^2 / (1 - g_3(4c^2))$,
where $g_3(s) = (1 + \sqrt{3s})\,e^{-\sqrt{3s}}$ is the radial generating
function for $\nu = 3$.

\section{Examples}\label{sec:examples}

We illustrate the framework with examples ordered following the
Layer~I$\to$II$\to$III narrative.

\subsection{Interval-censored data}\label{subsec:interval-censored}

One of the conceptual advantages of the present framework is that it
applies naturally when observations are not point values but distributional
objects.  Interval-censored data provide a central illustration.

Suppose the underlying latent variable $Y$ has distribution $P_\theta$,
represented distributionally by $(T_\theta, \varphi)$.  Instead of observing
exact values, one observes only intervals $I_i = [L_i, R_i]$ with
$Y_i \in I_i$.  The observation operator is
$\mathcal O_{I_i}(T_f) = \int_{I_i} f(x)\,dx$, and the pushforward model
assigns probability $p_\theta^I = P_\theta(Y \in I)$ to the event.

\paragraph{Estimating equations.}
Using the sinusoidal inference functional for a location parameter~$\mu$:
\[
\Psi(I, \mu)
= \int_I \sin(c(x - \mu))\,f_\mu(x)\,\varphi(x)\,dx.
\]
The estimating equation $n^{-1}\sum \Psi(I_i, \mu) = 0$ is well defined
even when the density $f_\mu$ is unavailable in closed form, provided the
weak characteristic function is known (since
$\Psi$ can be expressed through Fourier-analytic identities).

\paragraph{Interpretation.}
From the operator-theoretic perspective, interval censoring is not merely
``missing data'' but a change in observation mechanism: the observation
operator $\mathcal O_I$ replaces point evaluation $\delta_x$ by the
set-valued functional $\mathbf{1}_I$.  The information hierarchy
$I_{\mathrm{classical}} \geq I_I \geq G_\Psi$ quantifies the information
cost of censoring separately from the cost of using a specific inference
functional.

\subsection{Sinusoidal inference functions for heavy-tailed models}\label{subsec:sinusoidal}

\paragraph{Student $t$ location.}
Let $X = \theta + Y$ with $Y \sim t_\nu$ symmetric.  The sinusoidal
inference function
\[
\psi_c(x, \theta) = \sin(c(x - \theta)), \qquad c > 0,
\]
is bounded, requires neither a density nor finite moments, and is
unbiased by symmetry:
$\E_\theta[\sin(c(X - \theta))] = \mathrm{Im}(\phi_Y(c)) = 0$.

The empirical characteristic function
$z_n(u) = n^{-1}\sum e^{iuX_j}$ provides the estimating equation through
$\mathrm{Im}(e^{-iu\theta} z_n(u)) = 0$, or equivalently
\[
\frac{1}{n} \sum_{j=1}^n \sin(u(X_j - \theta)) = 0.
\]

If $z_n(u) = |z_n(u)|\,e^{i\alpha_n(u)}$, candidate estimators are
$\hat\theta_k = (\alpha_n(u) + 2k\pi)/u$, $k \in \mathbb Z$, and the
chosen estimator is the value closest to a robust pilot (e.g., the sample
median).

\paragraph{Sensitivity and variability.}
Using the characteristic function of $Y \sim t_\nu(0,1)$:
\begin{align*}
S_c(\theta) &= c\,g_\nu(c^2), \\
V_c(\theta) &= \tfrac{1}{2}(1 - g_\nu(4c^2)),
\end{align*}
where $g_\nu(s) = (\sqrt{\nu s})^{\nu/2}\,K_{\nu/2}(\sqrt{\nu s})\,
/\,[2^{\nu/2-1}\,\Gamma(\nu/2)]$ is the radial generating function and
$K_{\nu/2}$ is the modified Bessel function.  The Godambe information is
\[
G_c = \frac{2c^2\,g_\nu(c^2)^2}{1 - g_\nu(4c^2)}.
\]

\paragraph{Cauchy distribution.}
For $\nu = 1$ (Cauchy), $\phi_Y(u) = e^{-|u|}$, so $g_1(s) = e^{-\sqrt{s}}$
and
$G_c = 2c^2 e^{-2c}/(1 - e^{-4c})$.
The classical Fisher information is $I_{\mathrm{classical}} = 1/2$.  For
$c \approx 0.56$, the sinusoidal estimator achieves approximately 65\% of
the Fisher efficiency, while requiring no density evaluation.

\subsection{Weak moment-based inference}\label{subsec:weak-moment}

Given the kernel $\varphi$, the weak moments
${}^{(\varphi)}m_k(\theta) = \langle T_\theta, x^k\varphi \rangle$ are
defined for all $k \geq 0$.  The corresponding inference functions are
\[
\psi_k(x, \theta) = x^k\,\varphi(x) - {}^{(\varphi)}m_k(\theta),
\qquad k = 1, \ldots, K.
\]
The kernel factor $\varphi(x)$ is essential: it ensures consistency with the
distributional pairing
$\langle T_\theta, \psi_k(\cdot, \theta) \rangle
= \langle T_\theta, x^k\varphi \rangle - {}^{(\varphi)}m_k(\theta) = 0$.

For point observations, the empirical weak moment is
$n^{-1}\sum X_i^k\,\varphi(X_i)$, the kernel-weighted sample moment.

\begin{remark}[M-determinacy and kernel dependence]\label{rem:M-det}
For the weak moments to identify $\theta$, the weak moment problem must be
determinate.  By Theorem~5.1 of~\cite{LabouriauA1}, this is guaranteed for
positive Schwartz kernels with exponential tails (Carleman condition),
for kernels with exponential-type decay (Denjoy--Carleman), and for
Gaussian kernels (Hermite completeness).
\end{remark}

\subsection{Gaussian location (classical benchmark)}\label{subsec:gaussian}

For $X \sim \mathcal N(\theta, 1)$, the score function is
$\psi(x, \theta) = x - \theta$, yielding the sample mean.  The Godambe
information equals the Fisher information:
$G_\psi = I_{\mathrm{classical}} = 1$.  This confirms that the framework
recovers classical likelihood inference as a special case.

\subsection{Finite mixtures}\label{subsec:mixtures}

Consider a mixture model
$P_\theta = \pi P_1 + (1 - \pi) P_2$.  Although a density exists, the
score function may fail the unbiasedness condition, reflecting the
non-regular nature of mixture models.  Transform-based inference functions
bypass this difficulty, since the characteristic function of the mixture is
$\phi_\theta(u) = \pi\phi_1(u) + (1-\pi)\phi_2(u)$, which is well defined
and smooth.  A detailed analysis of inference functions in mixture models is
given in~\cite{Labouriau2022}.

\subsection{Simulation study}\label{subsec:simulation}

We compare three estimators in the model $X_i \sim t_\nu(\theta)$:
the sample mean, the sample median, and the sinusoidal (weak phase)
estimator.  We use $n = 100$ and $2000$ replications.

\begin{table}[htbp]
\centering
\caption{Monte Carlo results for the Student $t$ location model.}
\label{tab:t-sim}
\begin{tabular}{llrrrr}
\hline
$\nu$ & Estimator & Bias & Variance & MSE & MAD \\
\hline
1.5 & Mean   &  0.0122 & 2.4232 & 2.4221 & 0.2801 \\
1.5 & Median & -0.0043 & 0.0215 & 0.0215 & 0.0968 \\
1.5 & Weak   & -0.0038 & 0.0393 & 0.0393 & 0.1321 \\
\hline
3 & Mean   & -0.0017 & 0.0283 & 0.0283 & 0.1073 \\
3 & Median &  0.0005 & 0.0187 & 0.0186 & 0.0931 \\
3 & Weak   & -0.0018 & 0.0210 & 0.0210 & 0.0959 \\
\hline
\end{tabular}
\end{table}

The sample mean becomes unstable under heavy tails ($\nu = 1.5$); the weak
estimator remains stable and competitive with the median; performance
depends on the tuning parameter~$u$.

\begin{figure}[htbp]
\centering
\includegraphics[width=0.45\textwidth]{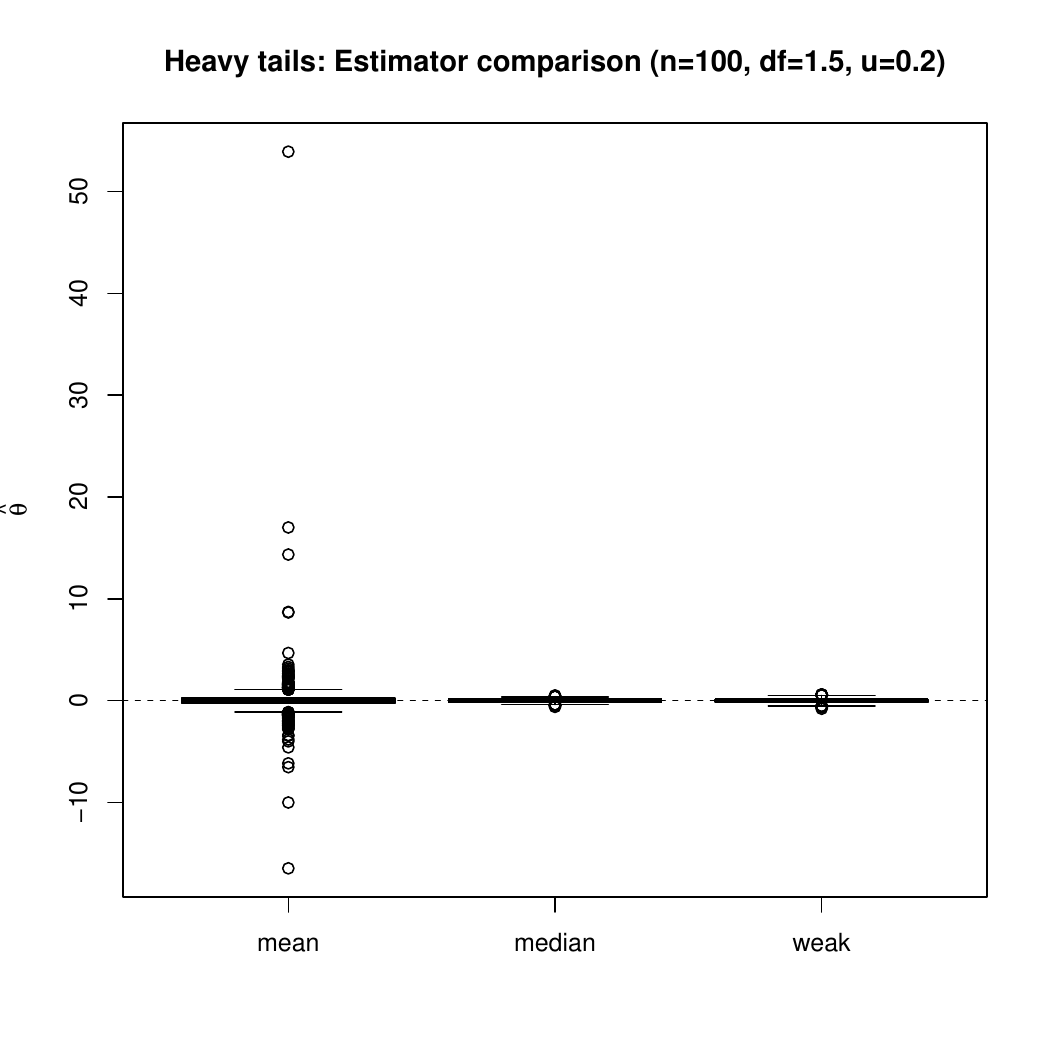}
\includegraphics[width=0.45\textwidth]{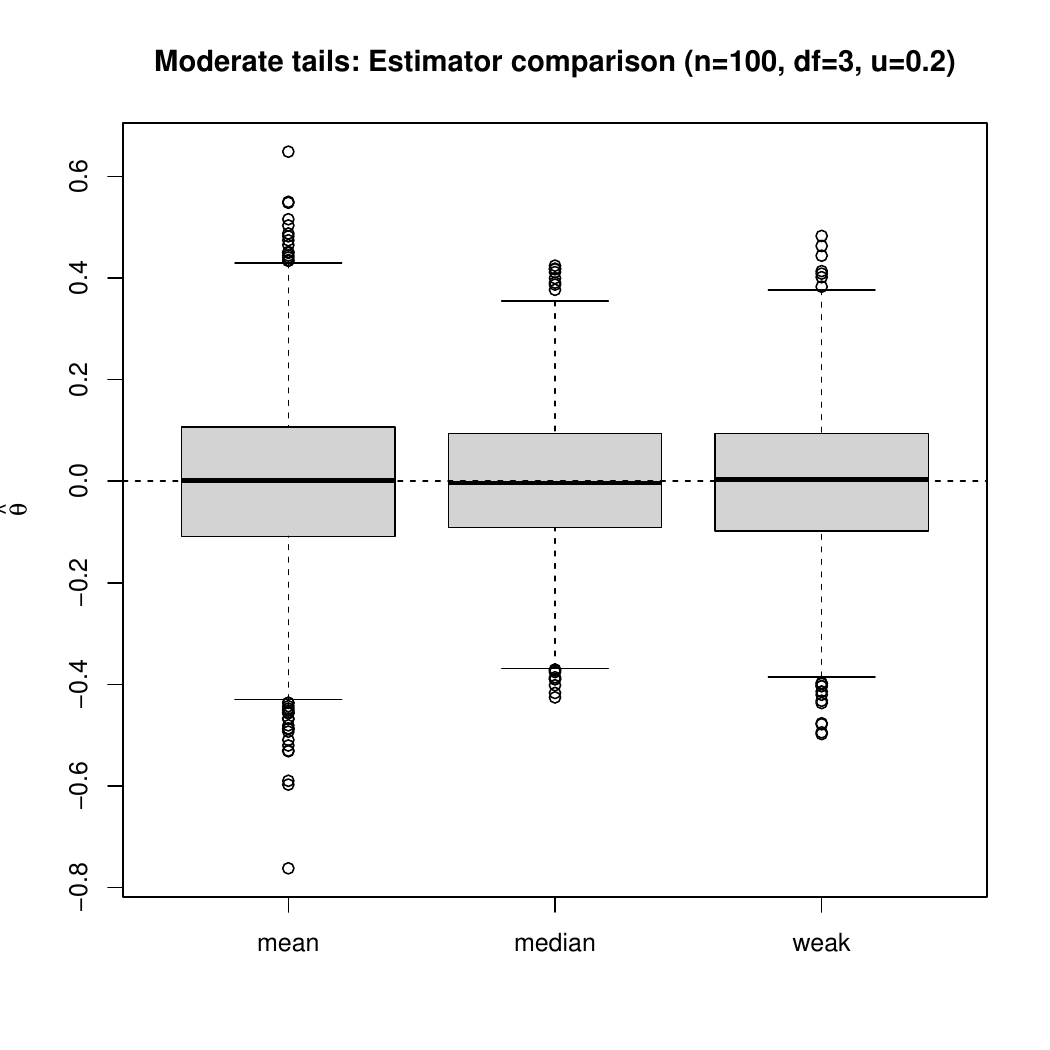}
\caption{Estimator comparison for Student $t$ models.}
\end{figure}

\section{Multivariate extension}\label{sec:multivariate}

\subsection{Multivariate distributional models and inference functions}

Let $\mathcal X \subseteq \R^k$ and
$\mathcal P = \{P_\theta : \theta \in \Theta \subseteq \R^p\}$.  Each
$P_\theta$ admits a distributional representation
$(T_\theta, \varphi)$ with $T_\theta \in \mathcal S'(\R^k)$ and
$\varphi \in \mathcal S(\R^k)$.  A vector-valued inference function
$\psi : \mathcal X \times \Theta \to \R^p$ defines the estimating equation
$n^{-1}\sum \psi(X_i, \theta) = 0$.

The sensitivity matrix is
$S(\theta) = -\E_{P_\theta}[\partial_\theta \psi(X, \theta)]
\in \R^{p \times p}$,
the variability matrix is
$V(\theta) = \E_{P_\theta}[\psi\,\psi^\top] \in \R^{p \times p}$,
and the Godambe information is
$G(\theta) = S^\top V^{-1} S$.

\begin{theorem}[Multivariate asymptotic normality]\label{thm:multi-AN}
Under the multivariate analogues of the conditions in
Theorem~\ref{thm:AN},
\[
\sqrt{n}\,(\hat\theta_n - \theta_0)
\;\Rightarrow\;
\mathcal N\!\left(0,\; S^{-1} V S^{-\top}\right).
\]
\end{theorem}

\subsection{Elliptically contoured distributions}\label{subsec:elliptical}

Let $X \in \R^k$ have an elliptically contoured distribution with
location~$\mu \in \R^k$, scatter matrix~$\Sigma \in \R^{k \times k}$
(positive definite), and characteristic function
$\phi_X(u) = e^{iu^\top\mu}\,\psi_{\mathrm{ec}}(u^\top\Sigma u)$
for a scalar generator $\psi_{\mathrm{ec}} : [0,\infty) \to \R$.
Let $a \in \R^k$ be a fixed direction; the scalar parameter of interest is
$\alpha = a^\top\mu$ (a linear functional of the location).  The
sinusoidal inference function
$\psi_v(x, \mu) = \sin(v^\top(x - \mu))$, $v \in \R^k$, is unbiased by
symmetry and bounded.

\paragraph{Gaussian benchmark.}
For the Gaussian case $\psi_{\mathrm{ec}}(s) = e^{-s/2}$, writing
$q = v^\top\Sigma v$, the sensitivity is
$S_v = (v^\top a)\,e^{-q/2}$ and the variability is
$V_v = \frac{1}{2}(1 - e^{-2q})$.
The Godambe information for $\alpha = a^\top\mu$ is therefore
\[
G_v
= \frac{S_v^2}{V_v}
= \frac{2(v^\top a)^2\,e^{-q}}{1 - e^{-2q}}.
\]
The classical Fisher information for~$\alpha$ is
$I_F = a^\top\Sigma^{-1}a$, and the asymptotic relative efficiency is
\[
\mathrm{ARE}(v)
= \frac{G_v}{I_F}
= \frac{2(v^\top a)^2\,e^{-q}}{(1 - e^{-2q})\,a^\top\Sigma^{-1}a}.
\]
With the optimal choice $v = c\,\Sigma^{-1}a$, so that
$q = c^2\,a^\top\Sigma^{-1}a = c^2 I_F$ and $v^\top a = c\,I_F$, this
simplifies to
$\mathrm{ARE}(c) = 2c^2 I_F\,e^{-c^2 I_F}/(1 - e^{-2c^2 I_F})$.
As $c \to 0$, $\mathrm{ARE}(c) \to 1$: the sinusoidal estimator
approaches full Fisher efficiency in the small-frequency limit.

\subsection{Multivariate Student $t$}\label{subsec:mult-t}

For $X \sim t_\nu(\mu, \Sigma)$ with characteristic function
$\phi_X(u) = e^{iu^\top\mu}\,g_\nu(u^\top\Sigma u)$,
where
\[
g_\nu(s)
= \frac{(\sqrt{\nu s})^{\nu/2}\,K_{\nu/2}(\sqrt{\nu s})}
       {2^{\nu/2-1}\,\Gamma(\nu/2)},
\]
the sinusoidal inference function $\psi_v(x, \mu) = \sin(v^\top(x-\mu))$
gives sensitivity $S_v = (v^\top e_\mu)\,g_\nu(v^\top\Sigma v)$ and
variability $V_v = (1 - g_\nu(4v^\top\Sigma v))/2$.

For joint inference on $(\mu, \Sigma)$ with $\nu$ known, one uses $m$
evaluation points $u_1, \ldots, u_m$ and the complex-valued inference
functions
$\psi_j(x; \mu, \Sigma) = e^{iu_j^\top x} - e^{iu_j^\top\mu}\,g_\nu(u_j^\top\Sigma u_j)$.
This provides $2m$ real equations for $p = k + k(k+1)/2$ parameters.
The optimal GMM weight $W = V^{-1}$ maximises the Godambe information.

The characteristic-function approach requires only evaluation of $K_{\nu/2}$
at $m$ points---no density evaluation or Fourier inversion---making it
computationally attractive in high dimensions.

\section{Nuisance parameters and inferential separation}\label{sec:nuisance}

\subsection{Setup}\label{subsec:nuisance-setup}

Consider $\mathcal P = \{P_{\alpha,\beta} : \alpha \in A, \beta \in B\}$
where $\alpha \in A \subseteq \R^k$ is the parameter of interest and
$\beta \in B$ is the nuisance parameter.  An inference function for
$\alpha$ is $\psi : \mathcal X \times A \to \R^k$ that depends on $x$ and
$\alpha$ but not on~$\beta$.

The unbiasedness condition is
\[
\E_{\alpha,\beta}[\psi(X, \alpha)] = 0
\qquad \text{for all } (\alpha, \beta) \in A \times B.
\]
This must hold uniformly over~$\beta$, which is more demanding than the
condition without nuisance parameters.

\subsection{Sensitivity, variability, and Godambe information}

Since $\psi$ does not depend on $\beta$, the sensitivity and variability
depend on the full parameter $(\alpha, \beta)$:
\begin{align*}
S_\psi(\alpha, \beta) &= -\E_{\alpha,\beta}[\partial_\alpha \psi(X, \alpha)], \\
V_\psi(\alpha, \beta) &= \E_{\alpha,\beta}[\psi(X, \alpha)\,\psi(X, \alpha)^\top].
\end{align*}
The Godambe information is
$J_\psi(\alpha, \beta) = S_\psi^\top V_\psi^{-1} S_\psi$,
and the asymptotic variance depends on $\beta$ through both $S$ and $V$.

\subsection{Orthogonality and the Godambe--Thompson projection}\label{subsec:orthogonality}

\subsubsection*{Nuisance tangent space: two-stage construction}

The semiparametric geometry of nuisance parameters involves two distinct
stages.  In the first (\emph{distributional}) stage, one differentiates the
distributional model with respect to the nuisance parameter; in the second
(\emph{Hilbert-space}) stage, one works in the $L^2$ space of the
pushforward model, where projections and orthogonality are well defined.

\paragraph{Stage~1 (distributional).}
When $P_{\alpha,\beta}$ admits a density, the classical nuisance tangent
space is the closed linear span
\begin{equation}\label{eq:nuisance-tangent}
\mathcal T_N(\alpha, \beta)
= \overline{\operatorname{span}}\!
\left\{
\partial_{\beta_j} \log p(\cdot\,; \alpha, \beta)
: j = 1, \ldots, \dim(\beta)
\right\}
\subset L^2_0(P_{\alpha,\beta}).
\end{equation}

In the distributional setting, densities may not exist, so one replaces
the nuisance score by distributional derivatives.  For each nuisance
direction $e_j$, the \emph{nuisance directional derivative}
$\dot T_j \in \mathcal S'(\R)$ satisfies
\[
\langle \dot T_j, g\varphi \rangle
= \lim_{\varepsilon \to 0}
\frac{\langle T_{\alpha,\beta+\varepsilon e_j} - T_{\alpha,\beta},\,
g\varphi \rangle}{\varepsilon}
\]
for all $g \in \mathcal A_\varphi$.  This derivative lives in
$\mathcal S'(\R)$ and is defined purely through the distributional pairing.

\paragraph{Stage~2 (Hilbert-space).}
The semiparametric geometry---projections, orthogonality, efficiency
bounds---operates in the Hilbert space $L^2_0(P_{\alpha,\beta}^{\mathcal O})$
of the \emph{pushforward model}.  When the pushforward model admits a
density, each distributional nuisance derivative $\dot T_j$ induces an
$L^2_0$-Riesz representative $h_j$ through the identity
$\E_{\alpha,\beta}[g(X)\,h_j(X)]
= \langle \dot T_j, g\varphi \rangle$ for all $g \in \mathcal A_\varphi$.
The distributional nuisance tangent space is then
\begin{equation}\label{eq:nuisance-tangent-dist}
\mathcal T_N(\alpha, \beta)
= \overline{\operatorname{span}}\!\{h_j : j = 1, \ldots, \dim(\beta)\}
\;\subset\; L^2_0(P_{\alpha,\beta}^{\mathcal O}),
\end{equation}
where the closure is in $L^2_0(P_{\alpha,\beta}^{\mathcal O})$.

This two-stage construction is important: the distributional derivatives
(Stage~1) define the nuisance directions without requiring densities, while
the Hilbert-space structure (Stage~2) provides the inner product needed for
the Bhapkar--Godambe projection below.  When both stages are available
(\ie the pushforward model has a density), the construction
agrees with~\eqref{eq:nuisance-tangent}.

\subsubsection*{Automatic orthogonality}

\begin{proposition}[Distributional orthogonality]\label{prop:orthogonality}
Let $\psi : \mathcal X \times A \to \R^k$ be a regular inference function
satisfying the nuisance-uniform unbiasedness condition.  Then each
component $\psi_i(\cdot, \alpha) \in \mathcal T_N^\perp(\alpha, \beta)$.
\end{proposition}

\begin{proof}
Since $\psi$ does not depend on $\beta$, differentiating the unbiasedness
condition with respect to $\beta_j$ gives
$\E_{\alpha,\beta}[\psi_i(X, \alpha)\,\partial_{\beta_j} \log p(X; \alpha, \beta)] = 0$.
By linearity and closure,
$\psi_i(\cdot, \alpha) \in \mathcal T_N^\perp$.
\end{proof}

\subsubsection*{Bhapkar--Godambe projection}

When the model admits a density $p(\cdot\,; \alpha, \beta)$, denote by
$U_\alpha(x) = \partial_\alpha \log p(x; \alpha, \beta) \in \R^k$
and $U_\beta(x) = \partial_\beta \log p(x; \alpha, \beta) \in \R^q$
the partial scores for $\alpha$ and $\beta$ respectively.
Given a quasi-inference function $\Phi(x; \alpha, \beta)$ (\eg the
partial score $U_\alpha$), the adjusted inference function is
\[
\psi^*(x, \alpha)
= \Phi(x; \alpha, \beta)
- \E[\Phi\,U_\beta^\top]\,(\E[U_\beta\,U_\beta^\top])^{-1}\,U_\beta(x).
\]
This is the $L^2_0(P_{\alpha,\beta})$-projection of $\Phi$ onto the
orthogonal complement $\mathcal T_N^\perp$ of the nuisance tangent space.
In the distributional setting, the projection takes the same form with
expectations interpreted as distributional pairings and $U_\beta$ replaced
by the Riesz representatives $h_j$
of~\eqref{eq:nuisance-tangent-dist}.

\subsection{Optimality with nuisance parameters}\label{subsec:nuisance-optimal}

\begin{theorem}[Optimal inference function]\label{thm:optimal-nuisance}
Among all regular inference functions $\psi$ satisfying the
nuisance-uniform unbiasedness condition,
$J_\psi(\alpha, \beta) \leq J_{\psi^*}(\alpha, \beta)$
in the L\"owner order, where $\psi^*$ is the efficient score projected
onto $\mathcal T_N^\perp$.  Explicitly, denoting by
$\Pi_{\mathcal T_N} : L^2_0(P_{\alpha,\beta}^{\mathcal O}) \to
\mathcal T_N$ the orthogonal projection onto the nuisance tangent space,
the optimal inference function is
\[
U^*_\alpha(x; \alpha, \beta)
= U_\alpha(x; \alpha, \beta)
- \Pi_{\mathcal T_N}[U_\alpha](x).
\]
Equality $J_\psi = J_{\psi^*}$ holds if and only if $\psi$ is
equivalent to $\psi^*$.
\end{theorem}

\begin{remark}[Existence]\label{rem:existence}
An optimal regular inference function does not always exist.  The
semiparametric bound is attained if and only if $\mathcal T_N^\perp$
contains a function equivalent to the efficient score (Labouriau, 1996,
Theorem~8).
\end{remark}

\subsection{Example: location-scale models}\label{subsec:locscale}

Let $X_i = \mu + \sigma Z_i$ with $Z_i \sim F_0$ symmetric.  The
sinusoidal inference function $\psi_c(x, \mu) = \sin(c(x - \mu))$ depends
only on $\mu$ and is automatically unbiased for all $\sigma > 0$
(since $\varphi_Z$ is real when $F_0$ is symmetric).  No Bhapkar--Godambe
adjustment is needed.

The Godambe information is
\[
J_{\psi_c}(\mu, \sigma)
= \frac{2c^2\,\varphi_Z(c\sigma)^2}{1 - \varphi_Z(2c\sigma)}.
\]

\begin{example}[Gaussian location-scale]\label{ex:gauss-locscale}
For $F_0 = N(0,1)$, $\varphi_Z(t) = e^{-t^2/2}$, and
$\mathrm{ARE}(u) = 2u\,e^{-u}/(1 - e^{-2u}) = u/\sinh(u)$ with
$u = c^2\sigma^2$.  The function $u/\sinh(u)$ is monotonically decreasing
from~$1$ (at $u = 0$) to~$0$ (as $u \to \infty$), so
$\mathrm{ARE} \to 1$ as $c \to 0$.  There is no interior optimum in the
Gaussian case: the sinusoidal estimator approaches full Fisher efficiency
in the small-frequency limit but at the cost of decreasing effective
signal (both $S_c$ and $V_c$ tend to zero as $c \to 0$).  For moderate
frequencies, \eg $u = 1$, the ARE is approximately~$0.85$.
\end{example}

\begin{example}[Student $t$ location-scale]\label{ex:t-locscale}
For the $t_\nu$ distribution, $\varphi_Z$ involves $K_{\nu/2}$, and the
Godambe information can be computed in closed form via Bessel functions.
For heavy-tailed distributions ($\nu$ small), the efficiency of the
sinusoidal estimator relative to MLE can be surprisingly high.
\end{example}

\subsection{Connection to semiparametric theory}\label{subsec:semiparametric}

When $\beta$ is infinite-dimensional, the framework becomes semiparametric.
The distributional nuisance tangent space
$\mathcal T_N$~\eqref{eq:nuisance-tangent-dist} is the closed linear span
of all nuisance-score directions in $L^2_0(P_{\alpha,\beta}^{\mathcal O})$.
Labouriau (1996) showed that the semiparametric Cram\'er--Rao bound
coincides with the inference-function bound if and only if
$\mathcal T_N^\perp = \mathcal F_{IA}$, where
$\mathcal F_{IA} = \bigcap_\beta \mathcal T_N^\perp(\alpha, \beta)$
denotes the \emph{admissible inference-function space}---the intersection
of all nuisance-orthogonal complements as $\beta$ varies.

The distributional framework provides a natural language for semiparametric
extensions: the two-stage construction of
Section~\ref{subsec:orthogonality} allows nuisance tangent directions to
be defined purely through distributional variation of $T_{\alpha,\beta}$
with respect to $\beta$ (Stage~1), even when densities are unavailable.
The Hilbert-space geometry needed for projections and efficiency bounds
is then provided by the pushforward model (Stage~2).

\begin{remark}[Current status]\label{rem:semiparametric-status}
The semiparametric theory in this section is developed at the level of a
mathematically plausible framework: the key structural ideas (two-stage
tangent space, Bhapkar--Godambe projection, orthogonality conditions) are
in place, and they agree with the classical theory when both apply.  A
fully rigorous semiparametric theory---with explicit functional-analytic
conditions, regularity of the tangent space as a function of the nuisance
parameter, and pathwise differentiability in the sense of
Bickel~\emph{et al.}~(1993)---requires additional technical development
and is the subject of ongoing work.
\end{remark}

\subsection{Inferential separation}\label{subsec:nonformation}

The classical notions of S-nonformation, I-nonformation, and L-nonformation
(Barndorff-Nielsen, 1978; J{\o}rgensen and Labouriau, 2012) admit
distributional formulations.

\begin{definition}[Weak S-nonformation]\label{def:weak-S}
The model is \emph{weakly S-nonformative} with respect to $\alpha$ if, for
every $g \in \mathcal S(\R)$,
$\E_\theta[g(X) | U]$ does not depend on $\alpha$ for all~$\beta$.
\end{definition}

\begin{definition}[Weak I-nonformation]\label{def:weak-I}
The model is \emph{weakly I-nonformative} with respect to $\alpha$ if the
conditional characteristic functions
$\{\E_\theta[e^{iuX} | V = v] : \alpha(\theta) = \alpha_0\}$
determine a saturated family.
\end{definition}

\begin{proposition}[Inference-function characterisation]\label{prop:IF-ancillarity}
$V$ is S-ancillary for $\alpha$ if and only if
$\E_\theta[\psi(X, \theta) | V] = 0$ for all $\theta$ and all regular
inference functions~$\psi$.
\end{proposition}

The sinusoidal inference function in symmetric location-scale models
achieves inferential separation automatically: it is orthogonal to the
nuisance tangent space without projection, reflecting the underlying
G-nonformation of the model.

\begin{remark}[Hierarchy of nonformation concepts]\label{rem:nonformation-hierarchy}
The classical hierarchy (B-sufficiency $\Rightarrow$ S-sufficiency
$\Rightarrow$ L-sufficiency; G-sufficiency $\Rightarrow$ L-sufficiency)
extends to the distributional setting.  The weak versions preserve this
hierarchy when classical versions are defined, and extend it to models
without densities.
\end{remark}

\section{Discussion}\label{sec:discussion}

This paper has developed a framework for statistical inference in which
observations are modelled by operators acting on distributional
representations of statistical models.  The main results are the
formulation of inference functionals composed with observation operators,
the asymptotic theory for regular inference functionals under mild
conditions, the information hierarchy
$I_{\mathrm{classical}} \geq I_{\mathcal O} \geq G_\Psi$ derived via
the convolution theorem, the extension to nuisance parameters and
inferential separation, and concrete examples in heavy-tailed and
censored-data settings.

The framework rests on the distributional foundation developed
in~\cite{LabouriauA1}, where distribution--kernel pairs, weak moments and
cumulants, determinacy theory, and a distributional central limit theorem
are established.  Weak moment estimation as a concrete inferential
methodology, with applications to robust estimation and generalised
method-of-moments, is developed in~\cite{LabouriauA2}.  The present paper
complements these works by providing a general operator-theoretic
framework---observation operators, inference functionals, pushforward
experiments, and the information hierarchy---within which the
methodological tools of~\cite{LabouriauA2} and the probabilistic
foundations of~\cite{LabouriauA1} can be placed in a unified theoretical
context.

Several directions for future work arise naturally from the
observation-operator formulation.  We describe those that seem most
immediate.

The first concerns \emph{random observation operators}.  In many
experimental settings, different observations pass through different
instruments with different response profiles.  This can be modelled by
allowing the observation operator to vary randomly:
$Y_i = \langle T, \varphi_{\xi_i} \rangle + \varepsilon_i$, where $\xi_i$
represents the state of the instrument.  Such a formulation connects
naturally with mixed models, since random effects in a generalised linear
mixed model may be reinterpreted as random perturbations of the
observation operator.  A concrete instance is the Behrens--Fisher problem,
in which two samples with different variances may be viewed not as
intrinsic heteroscedasticity but as observations obtained through
instruments with different resolutions.

A second direction involves \emph{operator-valued likelihood analogues}.
When observations are genuinely distributional---interval-valued,
convolutionally blurred, or transform-based---the classical likelihood
$\prod p_\theta(X_i)$ may be replaced by information functionals of the
observation operators.  For interval observations, this reduces to the
familiar interval-censoring likelihood $\prod \int_{I_i} f_\theta$; for
transform observations, spectral likelihood analogues provide a natural
alternative.  The theory of such operator-valued information functionals
remains to be developed.

Third, \emph{weak empirical processes} arise when the empirical
distributional object $\mathbb W_n$ and the fluctuation process
$\mathbb G_n = \sqrt{n}(\mathbb W_n - \E[\mathcal O])$ are viewed as
processes indexed by test functions, operator classes, or transform
families.  A Donsker-type theory for such processes---including weak
convergence of $\mathbb G_n$ to a Gaussian limit, tightness criteria for
distribution-valued processes, and entropy conditions for operator
classes---would extend the classical empirical process theory to the
distributional setting.  This requires specifying a topology on
$\mathcal S'$ for convergence, establishing a covariance structure for
the Gaussian limit, and identifying the appropriate function classes.
These are substantial open problems.

A fourth direction concerns \emph{conditional projection and mixed
models}.  Marginalisation over random effects in a mixed model is an
instance of applying an observation operator: the marginalisation operator
$\mathcal M_B$ integrates out the latent component and, in the density
case, reduces to $f^Y(y) = \int f^{Y,B}(y,b)\,db$.  Instead of
constructing the marginal likelihood, one may build inference functionals
directly from the projected distributional representation, using the
Bhapkar--Godambe projection to ensure insensitivity to the nuisance
component.

Fifth, the information hierarchy
$I_{\mathrm{classical}} \geq I_{\mathcal O} \geq G_\Psi$ defines a chain
of positive-definite bilinear forms on the tangent space
$T_\theta\Theta$, hence a chain of \emph{Riemannian metrics}.  The
classical Fisher metric corresponds to the full model; the metric
$g^{\mathcal O}$ captures the geometry of the pushforward experiment; and
$g^{\mathrm{Godambe}}_\Psi$ is the metric induced by a specific inference
functional.  Natural questions include: which observation operator
maximises $I_{\mathcal O}$ within a given class of feasible operators (an
optimal design problem), how the projected Godambe information relates to
the full metric via Riemannian submersion, and what role, if any, is
played by analogues of Amari's $\alpha$-geometry.  These geometric aspects
are developed in~\cite{LabouriauE} and ~\cite{LabouriauZ} .

Finally, the observation-operator formulation suggests a connection with
\emph{statistical inverse problems}.  In inverse problems, the quantity of
interest is observed only through a smoothing or otherwise ill-conditioned
operator; in the present framework, the observation operator plays an
analogous role, specifying the resolution at which the latent distributional
object is probed.  The inference functional can then be viewed as an
estimating procedure applied to a regularised image of the latent object,
and the observation cost $I_{\mathrm{classical}} - I_{\mathcal O}$
measures the price of regularisation.  Developing this connection
systematically may lead to adaptive kernel choices and scale-dependent
inference procedures.

Beyond these specific directions, the broader aim is to develop
distributional statistical models into a coherent framework encompassing
probabilistic foundations, inferential methodology, operator-theoretic and
geometric structure, and computational tools.  The most pressing open
problems, in our view, are the development of weak empirical process
theory, a rigorous semiparametric theory for distributional nuisance
parameters (cf.\ Remark~\ref{rem:semiparametric-status}), and the
extension of the asymptotic theory to stochastic observation operators.



\end{document}